\documentclass[12pt]{article}
\usepackage{amsmath,amssymb,amsthm,bbm,srcltx,graphicx,a4wide,xr,xargs, mathtools}
\usepackage[shortlabels]{enumitem}
\usepackage{csquotes}

\usepackage[usenames,dvipsnames]{color}
\usepackage[colorlinks]{hyperref}
\hypersetup{citecolor=ForestGreen}

\usepackage{cleveref}
\numberwithin{equation}{section}

\newcommand{\toi}{\to\infty}

\newcommand{\dto}{%
	\mathrel{\vbox{\offinterlineskip\ialign{%
				\hfil##\hfil\cr
				$\scriptstyle d$\cr
				$\longrightarrow$\cr
			}}}}
\newcommand{\wto}{%
	\mathrel{\vbox{\offinterlineskip\ialign{%
				\hfil##\hfil\cr
				$\scriptstyle w$\cr
				$\longrightarrow$\cr
			}}}}

\newcommand{\Pto}{%
	\mathrel{\vbox{\offinterlineskip\ialign{%
				\hfil##\hfil\cr
				$\scriptstyle \pr$\cr
				$\longrightarrow$\cr
			}}}}
			
\newcommand{\teind}{%
	\mathrel{\vbox{\offinterlineskip\ialign{%
				\hfil##\hfil\cr
				$\scriptstyle d$\cr
				$=$\cr
			}}}}			

\newcommand{\eind}{\stackrel{d}{=}}

\newcommand{\pr}{\mathbb{P}} 
\newcommand{\ex}{\mathbb{E}}

\newcommand{\ZZ}{\mathbb{Z}}

\newcommand\ind[1]{\mathbbm{1}{\left\{#1\right\}}}

\newcommand\1[1]{\mathbbm{1}_{#1}}

\newcommand{\N}{\mathbb{N}}
\newcommand{\R}{\mathbb{R}} 
 
\newcommand{\Z}{\mathbb{Z}}

\newcommand{\bQ}{\boldsymbol{Q}}

\newcommand{\bX}{\boldsymbol{X}}
\newcommand{\bY}{\boldsymbol{Y}}
\newcommand{\bZ}{\boldsymbol{Z}}

\newcommand{\bC}{\boldsymbol{C}}
\newcommand{\bTheta}{\boldsymbol{\Theta}}

\newcommand{\bx}{\boldsymbol{x}}

\newcommand{\bt}{\boldsymbol{t}}

\newcommand{\bi}{\boldsymbol{i}}
\newcommand{\bj}{\boldsymbol{j}}
\newcommand{\bk}{\boldsymbol{k}}
\newcommand{\bo}{\boldsymbol{0}}

\newcommand{\bz}{\boldsymbol{z}}

\newcommandx\sequence[3][2=\ZZ,3=t]{(#1_{#3})_{#3\in#2}}

\newtheorem{theorem}{Theorem}[section]
\newtheorem{lemma}[theorem]{Lemma}
\newtheorem{definition}[theorem]{Definition}
\newtheorem{corollary}[theorem]{Corollary}
\newtheorem{proposition}[theorem]{Proposition}

\theoremstyle{remark}
\newtheorem{remark}[theorem]{Remark}
\newtheorem{example}[theorem]{Example}

\newcommand{\tcb}[1]{#1}

\newcommand{\shift}[1]{\varphi_{#1}}
\newcommand{\e}{\mathrm{e}}
\newcommand{\dx}{\mathrm{d}}

\newcommand\xqed[1]{%
  \leavevmode\unskip\penalty9999 \hbox{}\nobreak\hfill
  \quad\hbox{#1}}
\newcommand\demo{\xqed{$\qed$}}

\newcommand{\tlim}{\textstyle\lim}
\newcommand{\tsum}{\textstyle\sum}

\newcommand{\espace}{\mathbb{E}_0}

\DeclarePairedDelimiter\abs{\lvert}{\rvert}%
\DeclarePairedDelimiter\norm{\lVert}{\rVert}%

\makeatletter
\let\oldabs\abs
\def\abs{\@ifstar{\oldabs}{\oldabs*}}
\let\oldnorm\norm
\def\norm{\@ifstar{\oldnorm}{\oldnorm*}}
\makeatother

\title{Palm theory for extremes of stationary regularly varying time series and random fields}
\author{ Hrvoje Planini\'c\thanks{Department of Mathematics, Faculty of science, University of Zagreb \newline  
    \hspace*{1.45em} Bijeni\v cka cesta 30, Zagreb, Croatia \newline    
    \hspace*{1.4em} \textit{E-mail:} planinic@math.hr}}
    
 \date{May 2, 2022}

\begin{document}

\maketitle

	\begin{abstract}
		The tail process $\bY=(Y_{\bi})_{\bi\in\Z^d}$ of a stationary regularly varying \tcb{random field} $\bX=(X_{\bi})_{\bi\in\Z^d}$ represents the asymptotic local distribution of $\bX$ as seen from its typical exceedance over a threshold $u$ as $u\to\infty$. Motivated by the standard Palm theory, we show that every tail process satisfies an invariance property called exceedance-stationarity and that this property, together with the spectral decomposition of the tail process, characterizes the class of all tail processes. We then restrict to the case when $Y_{\bi}\to 0$ as $|\bi|\to\infty$ and establish a couple of Palm-like dualities between the tail process and the so-called anchored tail process which, under suitable conditions, represents the asymptotic distribution of a typical cluster of extremes of $\bX$. The main message is that the distribution of the tail process is biased towards clusters with more exceedances. Finally, we use these results to determine the distribution of a typical cluster of extremes for moving average processes with random coefficients and heavy-tailed innovations.
\vskip 0.2cm

{\noindent\textit{Keywords}: tail process; \and regular variation; \and Palm theory; 
 \and stationarity; \and random fields; \and time series; \and moving averages}	
 \vskip 0.05 cm
 {\noindent\textit{MSC 2020}: 60G70; 60G10}	

		\end{abstract}

\section{Introduction}\label{sec:intro}

Consider a (strictly) stationary $\R$-valued random field $\bX=(X_{\bi})_{\bi\in\Z^d}$ with $d\in \N$; \tcb{if $d=1$ we say that $\bX$ is a time series}. A random field $\bY=(Y_{\bi})_{\bi\in\Z^d}$ is called the \textit{tail process} (\tcb{or \textit{tail field}})\footnote{\tcb{In the rest of the paper we will use the term \enquote{tail process} instead of \enquote{tail field}. This is mainly due to personal preference, but also since in almost all of the results the geometric structure of the index set is not important, i.e.\ almost all proofs for the case $d\geq 2$ are essentially the same as for the case $d=1$.}} of $\bX$, if $\pr(|Y_{\bo}|>1)=1$ and for all $m\in \N$, 
\begin{align}\label{eq:conv_to_tail}
 \pr \big( (u^{-1}X_{\bi})_{\bi \in \{-m,\dots, m\}^d} \in \cdot
\, \big| \, |X_{\bo}|> u \big) \wto 
\pr \big((Y_{\bi})_{\bi \in \{-m,\dots, m\}^d} \in \cdot \, \big) \, ,
\end{align}
as $u\toi$ in $\R^{(2m+1)d}$, where $\bo=(0,\dots,0)\in \Z^d$ and \enquote{$\wto$} denotes weak convergence. 

Existence of the tail process is equivalent to a regular variation property of $\bX$, and this is why such fields are called \textit{regularly varying}. The tail process of a regularly varying time series was introduced in Basrak and Segers~\cite{basrak:segers:2009}; for a detailed study, examples and further references see the recent monograph by Kulik and Soulier~\cite{kulik:soulier:2020}. Extension of the theory to random fields (i.e.\ to $d\geq 2$) was done by Wu and Samorodnitsky~\cite{wu:samorodnitsky:2020} and Basrak and Planini\'c~\cite{basrak:planinic:2020}. Very recently, Soulier~\cite{soulier:2021} considers even regularly varying continuous time stochastic processes.

Observe that  the tail process $\bY$ by definition always has an \textit{exceedance-point} at the origin in the sense that $|Y_{\bo}|>1$ almost surely. Moreover, one can informally say that it represents the asymptotic (local) distribution of the stationary field $u^{-1}\bX$ as $u\toi$ as seen from its \textit{typical} exceedance-point (located at the origin).\footnote{The term \enquote{typical} used here (and throughout the paper) has the strong motivation of being \enquote{chosen uniformly at random} (or simply \enquote{randomly chosen}) only under certain dependence assumptions on $\bX$, see  \Cref{sec:randomized_origin} and \Cref{exa:MA_counter} below, and also \cite{last:thorisson:2011, thorisson:2018} for a general discussion.} As a consequence, $\bY$ is in general not stationary. Still, the stationarity of $\bX$ yields a certain invariance property of the tail process $\bY$ given in (\ref{eq:time-change1}) below. We will characterize this invariance property as \textit{exceedance-stationarity}, see \Cref{sub:exc_stat} and in particular \Cref{thm:Y_is_point_stat}.
Intuitively, a process is exceedance-stationary if its distribution is the same as seen from any of its exceedance-points.
Exceedance-stationarity of the tail process can be seen as the extreme value analogue of the \textit{point-stationarity} property of the Palm distribution of a stationary point process.
The latter concept was introduced by Thorisson~\cite[Chapter 9]{thorisson:2000}, see also~\cite[Section 4.4.9]{chiu:2013} for a brief description written by G. Last.

Every tail process admits a so-called spectral decomposition: $|Y_{\bo}|$ is independent of the \textit{spectral (tail) process} $\bTheta=(\Theta_{\bi})_{\bi \in \Z^d}$ defined by
$$\Theta_{\bi}=|Y_{\bo}|^{-1} Y_{\bi} \, , \, \bi \in \Z^d \, ,$$ 
and moreover satisfies $\pr(|Y_{\bo}|>y)=y^{-\alpha}$, $y\geq 1$, for some $\alpha>0$; we say that $|Y_{\bo}|$ is $\mathrm{Pareto}(\alpha)$-distributed. The spectral process also satisfies a certain (strange-looking) property given in (\ref{eq:time-change0}) below. In the case $d=1$ it is often referred to as the \textit{time-change formula}.
It is known that this property characterizes the class of all (spectral) tail processes, see the recent works by Jan{\ss}en~\cite{janssen:2019}, Dombry~et~al.~\cite{dombry:hashorva:soulier:2018} and Hashorva~\cite{hashorva:2020}. \tcb{We will give a new, arguably more probabilistic, interpretation of the time-change formula (\ref{eq:time-change0}) by showing that, under the spectral decomposition, it is equivalent to the exceedance-stationarity of the tail process, see \Cref{prop:TCF_another_view} for details.}

\tcb{Further, one often considers fields $\bX$ whose exceedances over a high threshold group into independent and identically distributed (i.i.d.)\ \textit{clusters} as the threshold size tends to infinity, see (\ref{eq:Poisson_approx_for_clusters}) for a precise statement in the case of time series. As originally observed by Davis and Hsing~\cite[Theorem 2.5]{davis:hsing:1995}, the common distribution of these clusters equals the asymptotic distribution of a relatively small block of observations of $\bX$ conditionally on having at least one exceedance over the threshold;
see (\ref{eq:anchored_process_cluster}) for a precise description of this convergence. Basrak and Segers~\cite{basrak:segers:2009} refer to this distribution simply as the (asymptotic) distribution of a cluster of extremes (of $\bX$).  However, we will refer to it as the distribution of a \textit{typical} cluster of extremes since the distribution of the tail process, on the other hand, can be seen as the distribution of a cluster of extremes containing a {typical exceedance-point} of $\bX$ which is located at the origin; this is justified by  (\ref{eq:tail_process_cluster}) below.}

%

Basrak and Segers~\cite[Proposition 4.2]{basrak:segers:2009} showed that, when $d=1$, the distribution of a typical cluster equals the distribution of the tail process $\bY$ but conditionally on $\sup_{i\leq -1} |Y_{i}|\le 1$, or in words, conditionally on the  exceedance-point at the origin being the first exceedance-point in the corresponding cluster. Recently, \cite{basrak:planinic:2020} introduced the notion of \textit{anchoring} and, for general $d\in \N$, showed  that  instead of choosing the first exceedance-point as the \textit{anchor}, one can choose any other exceedance-point of the cluster as long as this choice is made in a translation covariant way. The resulting process was called the \textit{anchored tail process}, \tcb{and its distribution is thus the distribution of a typical cluster of extremes of $\bX$; see (\ref{eq:anchored_process_cluster}).}

The main goal of this paper is to further clarify the relationship between the tail process and the anchored tail process (i.e.\ the typical cluster). In particular, we will establish a certain Palm-like duality which implies that, up to a random shift, the tail process is a \textit{size-biased} version of the typical cluster, where by \textit{size} we mean the number of exceedance-points in the cluster, see \Cref{prop:Palm1}. We now give an intuitive explanation of this phenomenon. 
Since the clusters are assumed i.i.d. (and with the expected number of exceedance-points being finite), 
the tail process is in fact the cluster containing and being centered around an exceedance-point which is chosen \textit{uniformly} among all of the exceedance-points of $\bX$; see (\ref{eq:randomized_origin})
for the precise meaning of this statement (in the case $d=1$).
The law of large numbers now implies that this cluster cannot be typical -- the probability that it contains exactly $k\in \N$ exceedance-points is proportional to $k$ times the probability of the same event for the typical cluster. Moreover, the chosen exceedance-point located at the origin is chosen uniformly over all of the exceedance-points in the cluster; this intuition is essentially borrowed from Thorisson~\cite[p.\ 70]{thorisson:2000}.  The following simple example should make this idea clear.


\begin{example}\label{exa:MA_intro}
Let $(Z_{i})_{i\in \Z}$ be an i.i.d.\ sequence of $\mathrm{Pareto}(\alpha)$-distributed random variables for some $\alpha>0$, and let $(\varepsilon_i)_{i\in \Z}$ be an another independent i.i.d.\ sequence of Beornoulli random variables such that $\pr(\varepsilon_0=0)=\pr(\varepsilon_0=1)=\tfrac12$. Define a stationary  moving average process $\bX=(X_i)_{i\in \Z}$ by
\begin{align}\label{eq:MA_intro}
X_i=Z_i+ \varepsilon_i Z_{i-1} \, , \, i\in \Z \, .
\end{align} 

\begin{figure}[!h]
\centering
\includegraphics[scale = 0.5]{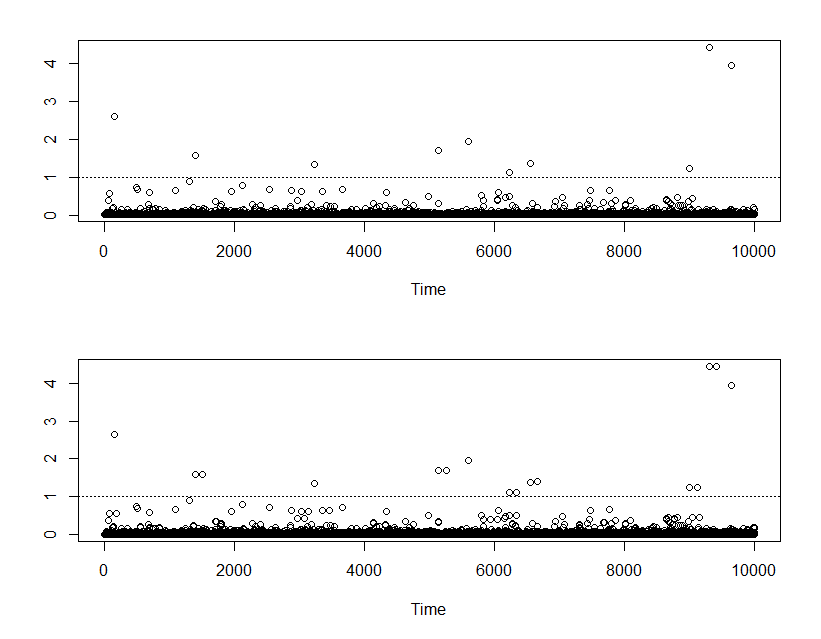}
\caption{One realization of the scaled i.i.d.\ sequence $Z_1/u,\dots, Z_n/u$ of $\mbox{Pareto}(\alpha)$-distributed random variables (top) and of the corresponding scaled moving average process $X_1/u,\dots, X_n/u$ from (\ref{eq:MA_intro}) (bottom) for $n=10^4$, $u=260$, and $\alpha=1.2\,$.}
\label{fig:MA_intro}
\end{figure}

The tail process $\bY=(Y_{i})_{i\in \Z}$  of $\bX$ exists (thus $\bX$ is regularly varying) and is given by  $Y_i=0$ for $|i|\geq 2$ and
\begin{align*}
(Y_{-1}, Y_0,Y_1) = \begin{cases}
(0,Y_0,0) \, , & \text{w.p.} \; \frac13 \\
(0,Y_0,Y_0) \, , & \text{w.p.} \; \frac13 \\
(Y_0,Y_0,0) \, , & \text{w.p.} \; \frac13 \, ,
\end{cases}
\end{align*}
where, recall, $Y_0=|Y_0|$ is $\mathrm{Pareto}(\alpha)$-distributed; see \Cref{exa:MA}. 

Inutitively, $Z_i$'s which exceed a high threshold come in isolation and each such exceedance generates  with equal probability (depending on the corresponding $\varepsilon_{i+1}$ being 0 or 1) one or two consecutive and asymptotically equal exceedances of $\bX$, see Figure \ref{fig:MA_intro} for an illustration. Thus, by choosing an exceedance of $\bX$ {uniformly at random} one is two times more likely to end up in a cluster with two consecutive exceedances. 
See also \Cref{exa:MA} below for more details on this model as well on a class of slightly more complicated models.
\demo
\end{example}

We note that the stated duality as well as all other results regarding the relationship between the tail process and the anchored tail process in \Cref{subs:anchor} are deduced  assuming only that the tail process $\bY=(Y_{\bi})_{\bi\in \Z^d}$ has finitely many exceedance-points, i.e.\ that the set $\{\bi : |Y_{\bi}|>1\}$ is a.s.\ finite;  due to the spectral decomposition this is equivalent to $\tlim_{|\bi|\to\infty}|Y_{\bi}|= 0$ a.s.\ 
where $|\bi|=\max_{k=1,\dots, d} |i_k|$ for $\bi=(i_1,\dots, i_d)\in \Z^d$. Thus, there are no additional assumptions on the underlying regularly varying random field. In particular, we establish a similar duality between the spectral tail process and a normalized version of the anchored process, see \Cref{prop:exc_stat_representation}, and use these two dualities to provide a new characterization of the class of so-called representatives of the anchored (spectral) tail process as well as of the class of all tail processes with finitely many exceedance-points, see \Cref{subs:palm_tail,subs:palm_spectral}  for details.  \tcb{In \Cref{sub:randomized_origin} we give formal meaning to statements which relate the tail and anchored tail process to clusters of extremes of $\bX$. In particular, we show that under the usual dependence assumptions, a \enquote{typical} cluster or exceedance-point is the same as a \enquote{randomly chosen} cluster or exceedance-point, see \Cref{prop:randomized_origin_cluster}.}

%
%
%

Finally, we apply our results on  a large and important family of regularly varying  time series and random fields, namely the class of infinite order moving average processes with random coefficients and regularly varying innovations, see \cite[Chapter 15]{kulik:soulier:2020}. In particular, we determine the distribution of the anchored tail process and provide an expression for the so-called candidate extremal index, both of which seem to be new, see \Cref{prop:MA_Q}.

The paper is organized as follows:  In \Cref{sec:exc_stat} we deal with the concept of exceedance-stationarity.  In \Cref{subs:anchor} we restrict our attention to tail processes with finitely many exceedance-points and prove our main results. In \Cref{sec:MA} we focus on moving average processes with random coefficients and regularly varying innovations. \Cref{sec:appendix} contains a couple of postponed proofs.

\begin{remark}
For notational convenience we only consider $\R$-valued random fields and time series. All the result easily extend to the general $\R^n$-valued case by replacing the absolute value $|\,\cdot \,|$ with an arbitrary norm on $\R^n$. \end{remark}


\section{Exceedance-stationarity of the tail process}
\label{sec:exc_stat}
Let the space $\R^{\Z^d}$ be equipped with the product $\sigma$-algebra, and for $\bk\in \Z^d$, let $\shift{\bk}: \R^{\Z^d} \to \R^{\Z^d}$ be the shift operator defined by
\begin{align}\label{eq:shift_operator}
\shift{\bk}	 \bx = (x_{\bi+\bk})_{\bi \in \Z^d} \, , \, \bx=(x_{\bi})_{\bi \in \Z^d} \in \R^{\Z^d} \, .
\end{align}
In words, $\shift{\bk}$ moves the origin to $\bk$. If $\bY=(Y_{\bi})_{\bi \in \Z^d}$ is the tail process of a stationary random field $\bX$, the stationarity of $\bX$ and convergence in (\ref{eq:conv_to_tail}) directly imply that 
\begin{equation}\label{eq:time-change1}
\ex\left[h(\shift{\bk}\bY)\ind{|Y_{\bk}|>1}\right]=\ex\left[h(\bY)\ind{|Y_{-\bk}|>1}\right] \, ,
\end{equation}
for every measurable function $h:\R^{\Z^d}\to \R_+$ and all $\bk\in \Z^d$;
see \cite[Lemma 3.3]{basrak:planinic:2020}. 


We start by characterizing the  property (\ref{eq:time-change1}) using the concept of exceedance-stationarity, an extreme value analogue of the concept of point-stationarity introduced by Thorisson~\cite[Chapter 9]{thorisson:2000}. Next we show that, under the spectral decomposition, (\ref{eq:time-change1}) holds if and only if the corresponding spectral tail process satisfies the so-called \enquote{time-change formula} given in (\ref{eq:time-change0}) below. Together, these two results give a new and arguably more probabilistic view on the \enquote{time-change formula} as well as on the  properties of tail processes.

\subsection{Bijective exceedance-maps}\label{sub:exc_stat}
In order to formally define the concept of exceedance-stationarity we first define the notion of a bijective exceedance-map. 
We will say that $\bk\in \Z^d$ is an \textit{exceedance-point} of $\bx=(x_{\bi})_{\bi \in \Z^d}\in \R^{\Z^d}$ if $|x_{\bk}|>1$. Let
$$\e(\bx)=\{\bk \in \Z^d : |x_{\bk}|>1\}\, ,$$ 
be the set of all exceedance-points of $\bx$, and set 
$$\R^{\Z^d}_0=\{\bx \in \R^{\Z^d}: \bo\in \e(\bx) \} \, .$$ 
Note that  $\pr(\bY \in\R^{\Z^d}_0)=1$ holds by definition for every tail process $\bY$. 

For a Borel subset $B\subseteq \R^{\Z^d}_0$, a measurable map $\tau : B \to \Z^d$ is called an \textit{exceedance-map} on $B$ if 
\begin{align*}
\tau(\bx)\in \e(\bx) 
\end{align*}
for all $\bx \in B$. An exceedance-map $\tau$ is called \textit{bijective} on $B$ if the associated map $\tau(\bx, \cdot \,)$ defined by
\begin{align}\label{eq:ass_exc_map}
\tau(\bx, \bk)= \bk + \tau(\shift{\bk}\bx) \, , \, \bk \in \e(\bx) \, ,
\end{align} 
is a bijection on $\e(\bx)$ for all $\bx \in B$. 
Note that we implicitly assume that the domain $B$ is invariant to shifting the origin to any of the exceedance-points, that is, $\bx\in B$ implies $\shift{\bk}\bx\in B$ for all $\bk\in\e(\bx)$.


\begin{example}\label{exa:nth_point_to_the_right}
Assume that $d=1$ and consider the space 
$$B=\{\bx\in \R^{\Z}_0 : \e(\bx)=\infty \text{ and } \inf \e(\bx)=-\infty\} \, . $$ 
\tcb{For every $\bx\in B$ and all $n\in \Z$, set 
\begin{align*}
\tau_{n}(\bx)=i_{n} \, , 
\end{align*} 
where $(i_m)_{m\in \Z}$ are all the exceedance-points of $\bx$ labeled such that $i_0=0$ and 
$$\cdots <i_{-1}<i_0<i_1<\cdots\, .$$}
In words, if $n\ge 1$, $\tau_n(\bx)$ and $\tau_{-n}(\bx)$ are the $n$th exceedance-points of $\bx$ to the right and left of 0, respectively. 
For all $k\in \e(\bx)$, $\tau_n(\bx,k)$ and $\tau_{-n}(\bx,k)$ are then the $n$th exceedance-points of $\bx$ to the right and left of $k$, respectively. Observe that $\tau_{0}(\bx,k)=k$ for all $k\in \e(\bx)$. Consequently, $\tau_n$ is a bijective exceedance-map on 
$B$
for all $n\in \Z$.
\demo
\end{example}


For any exceedance-map $\tau$ on $B\subseteq \R^{\Z^d}_0$ denote by $\shift{\tau}: B \to B$ the associated \textit{exceedance-shift} defined by
\begin{align}\label{eq:exc_shift}
\shift{\tau} \bx = \shift{\tau(\bx)} \bx , \;  \bx \in B \, ,
\end{align}
i.e., $\shift{\tau}$ shifts the origin to the exceedance-point chosen by $\tau$. It can be shown that an exceedance-map $\tau$ is bijective on $B$ if and only if $\shift{\tau}$ is a bijection on $B$ but we omit the details, cf.~\cite{heveling:thesis}. Finally, for a random element $\bY$ of $\R^{\Z^d}_0$, call an exceedance-map $\bY$-bijective if it is bijective on a set containing the support $\bY$.

\begin{definition}\label{def:point_stat}
A random element $\bY$ of $\R^{\Z^d}_0$ is \text{exceedance-stationary} if for every $\bY$-bijective exceedance-map $\tau$,
\begin{align}\label{eq:point_stat}
\shift{\tau} \bY \eind \bY \, \, \text{ in  } \R^{\Z^d}_0 \, ,
\end{align}
where $"\eind"$ denotes equality in distribution.  
\end{definition}

\begin{theorem}\label{thm:Y_is_point_stat}
A random element $\bY=(Y_{\bi})_{\bi \in \Z^d}$ of $\R^{\Z^d}_0$ is \text{exceedance-stationary} if and only if (\ref{eq:time-change1}) holds for every measurable $h:\R^{\Z^d}\to \R_+$ and all $\bk\in \Z^d$. In particular, tail processes of stationary random fields are exceedance-stationary.
\end{theorem}

%
%
%
%
%
%
%

Before the proof we give several remarks, and examples of bijective, as well as, non-bijective exceedance-maps. First, let $\abs{\e(\bx)}$ denote the cardinality of $\e(\bx)$ and let 
\begin{align*}
\espace=\{\bx\in \R_0^{\Z^d} : \abs{\e(\bx)}<\infty\} \, ,
\end{align*}
be the space of all elements of $\R_0^{\Z^d}$ with finitely many exceedance-points.

\begin{example}
\label{exa:finitely_many_exc}

\tcb{For every $\bx\in \espace$ and all $n\in \Z$, define $\tau_n(\bx)$ as follows.}
Denote $N=|\e(\bx)|$ and let $\bi_0,\dots, \bi_{N-1}$ be the elements of $\e(\bx)$ ordered lexicographically. Let $k\in \{0,\dots, N-1\}$ be such that $\bi_k=\bo$. Similarly as in \Cref{exa:nth_point_to_the_right}, set
\begin{align*}
\tau_n(\bx)=\bi_{k+n \; (\mathrm{mod}\; N)} \, .
\end{align*}
For all $\bx\in \espace$ and all $n\in \Z$ the associated map $\tau_n(\bx, \cdot\,)$ is clearly bijective on $\e(\bx)$. 
Observe that one can replace the lexicographic order with an arbitrary \textit{group} order on $\Z^d$, that is, a total order $\preceq$ which is translation invariant: $\bi \preceq \bj$ implies  $\bi+\bk \preceq \bj+\bk$ for all $\bi,\bj,\bk \in \Z^d$. 
\demo
\end{example}

\begin{remark}
In fact, if $\bY$ is a random element of $\espace$ (which is the case we are most interested in, see \Cref{subs:anchor}), it is enough to check (\ref{eq:point_stat}) for $\tau=\tau_n$ for all $n\in \Z$; this can be shown using similar arguments as in the proof of \Cref{thm:Y_is_point_stat} below. Since $\shift{\tau_n}\circ \shift{\tau_{m}}=\shift{\tau_{m+n}}$ for all $m,n\in \Z$, exceedance-stationarity of $\bY$ is thus equivalent to 
\begin{align*}
\shift{\tau_1} \bY \eind \bY \, .
\end{align*}
The same is true if $\sup \e(\bY)=\infty$ and $\inf \e(\bY)=-\infty$ a.s.\ with $\tau_1$ defined in \Cref{exa:nth_point_to_the_right}.
\demo
\end{remark}
\begin{example} 
Assume that $d=1$ and consider the space $B=\{\bx \in \R^{\Z}_0 : \sup \e(\bx)=\infty \, , \, \inf \e(\bx)>-\infty \}$. \tcb{For every $\bx\in B$ and all $n\in \N$, define $\tau_n(\bx)$ as follows.}  Let $i_0< i_1<i_2<\cdots \,$, be ordered elements of $\e(\bx)$ and set 
\begin{align*}
\tau_n(\bx)=i_{k+n} \, , \,
\end{align*}
where $k\in \N$ is such that  $i_k=0$. For all $\bx\in B$ and all $n\in \N$ the associated map $\tau_n(\bx, \cdot \,)$ will not be surjective on $\e(\bx)=\{i_0,i_1,\dots\}$ since $\tau_n(\bx, i_m)\neq i_0=\min \e(\bx)$ for all $m\geq 0$. See the proof of \Cref{thm:Y_is_point_stat} for a construction of a bijective exceedance-map on such $\bx$'s.
\demo
\end{example}

\begin{example}
For $\bx\in \R^{\Z^d}_0$ such that $\e(\bx)\neq \{\bo\}$, call $\bi\in \e(\bx)$ the \textit{nearest exceedance-point} of $\bk\in \e(\bx)$, if $\bi\neq \bk$ and $\|\bk-\bi\|=\min\{\|\bk-\bj\| : \bj\in \e(\bx)\setminus\{\bk\}\}$, where $\|\, \cdot \,\|$ is the Euclidean norm; if there are more such $\bi$'s, take the smallest w.r.t.\ the lexicographic order. In the trivial case $\e(\bx)= \{\bo\}$, let $\bo$ be its own nearest exceedance-point.

Now for all $\bx\in \R^{\Z^d}_0$ set $\tau(\bx)=\bi$ if $\bi$ is the nearest exceedance-point of $\bo$ and $\bo$ the nearest exceedance-point of $\bi$. In all other cases, set $\tau(\bx)=\bo$. The resulting exceedance-map $\tau$ is easily seen to be bijective on $\R^{\Z^d}$. Observe that by simply setting $\tau$ to always be the nearest exceedance-point of $\bo$ does not yield a bijective exceedance-map since if e.g.\ $\bx\in \R^{\Z^d}_0$ has a pair of exceedance-points with the same nearest exceedance-point, the associated map $\tau(\bx,\cdot\,)$ will not be injective. \demo
\end{example}

\begin{remark}\label{rem:mecke_tail}
For a random element $\bY=(Y_{\bi})_{\bi \in \Z^d}$ of $\R_0^{\Z^d}$, it is easily seen that the property (\ref{eq:time-change1}) holds for every measurable $h:\R^{\Z^d}\to \R_+$ and all $\bk\in \Z^d$ if and only if 
\begin{align}\label{eq:mecke}
\ex\big[\tsum_{\bk\in \e(\bY)} g(-\bk , \shift{\bk}\bY)\big] = \ex\big[\tsum_{\bk\in \e(\bY)} g(\bk , \bY)\big] \, 
\end{align}
holds for all measurable $g:\Z^d \times \R^{\Z^d}\to \R_+$, see \Cref{sec:appendix} for details. Equation (\ref{eq:mecke}) can be seen as an analogue of the so-called Mecke's integral equation, see e.g.~\cite[Equation (4.11)]{heveling:last:2005}, which is known to be equivalent to the point-stationarity property of a point process on $\R^d$, see \cite[Theorems 4.2.2 and 6.2.3]{heveling:thesis}. \demo

\end{remark}

\begin{remark}\label{rem:tail_is_Palm}
\tcb{Exceedance-stationarity of the tail process $\bY$ can be understood as a consequence of $\bY$ being an asymptotic Palm version of the underlying random field $\bX$ -- it is the asymptotic (local) distribution of $u^{-1}\bX$ as $u\toi$ conditionally on $u^{-1}\bX$ having an exceedance-point at the origin. Palm distribution are often introduced through the \textit{refined Campbell theorem}, see e.g.~\cite[Theorem 9.1]{last:penrose:2017}. In our context, a Campbell theorem could be stated as follows. }

\tcb{Let $f:[0,1]^d\times \R^{\Z^d} \to \R_+$ be an arbitrary measurable, continuous and bounded function; $\R^{\Z^d}$ is considered w.r.t.\ the product topology and corresponding Borel $\sigma$-algebra. Further, let $(c_n)_{n\in\N}$ be a sequence of positive real numbers such that as $n\toi$, $c_n\toi$ and
\begin{align*}
\ex\Big[\tsum_{\bi\in \{1,\dots, n\}^d} \1{\{|X_{\bi}|> c_n\}}\Big] = n^d\pr(|X_{\bo}|>c_n) \to \tau \in (0,\infty) \, ,
\end{align*}
i.e., $\tau$ is the asymptotic expected number of exceedance-points of $c_n^{-1}\bX$ in $\{1,\dots, n\}^d$.
Then as $n\toi$, stationarity of $\bX$ and the definition of the tail process in (\ref{eq:conv_to_tail}) imply
\begin{align}\label{eq:campbell_tail}
\ex\Big[\tsum_{\bk \in \{1,\dots, n\}^d} f(\bk/n, \shift{\bk}\bX / c_n) \1{\{|X_{\bk}|>c_n\}}\Big] \to \tau \int_{[0,1]^d} \ex\big[f(\bt, \bY) \big] \dx \bt \, .
\end{align}
The proof of (\ref{eq:campbell_tail}) can be found in \Cref{sec:appendix}. \demo}
\end{remark}

\begin{proof}[Proof of \Cref{thm:Y_is_point_stat}]

We first prove sufficiency. Let $\tau$ be a bijective exceedance-map on $B$ and assume that (\ref{eq:time-change1}) holds. 
Observe that, by construction, the mapping $\tau(\, \cdot \, , \cdot \,)$ satisfies the covariance property 
\begin{align}\label{eq:covar_property_tau}
\tau(\shift{\bk}\bx,\bk'-\bk) = \tau(\bx, \bk')- \bk \, , \;\; \bx\in B, \, \bk,\bk'\in \e(\bx) \, .
\end{align}

Take an arbitrary measurable function $h : \R^{\Z^d}_0 \to \R_+$. Then,
\begin{align*}
\ex[h(\shift{\tau}\bY)] &= \sum_{\bk \in \Z^d} \ex[h(\shift{\bk}\bY) \ind{\tau(\bY)=\bk}] && \\
&=  \sum_{\bk \in \Z^d} \ex[h(\shift{\bk}\bY) \ind{\tau(\bY, \bo)=\bk, |Y_{\bk}|>1}] && \text{since necessarily} \; \bk\in\e(\bY)  \\
& = \sum_{\bk \in \Z^d} \ex[h(\shift{\bk}\bY) \ind{\tau(\shift{\bk}\bY,-\bk)=\bo, |Y_{\bk}|>1}] && \text{use} \; (\ref{eq:covar_property_tau})  \\
& = \sum_{\bk \in \Z^d} \ex[h(\bY) \ind{\tau(\bY,-\bk)=\bo, |Y_{-\bk}|>1}] && \text{use} \; (\ref{eq:time-change1})  \\
& =  \ex\Big[h(\bY) \cdot \sum_{\bk' \in \Z^d}\ind{\tau(\bY,\bk')=\bo, |Y_{\bk'}|>1}\Big] && \bk'=-\bk   \\
&  =  \ex\Big[h(\bY) \cdot \sum_{\bk' \in \e(\bY)}\ind{\tau(\bY,\bk')=\bo}\Big] &&   \\ 
& = \ex[h(\bY) \cdot 1 ]= \ex[h(\bY)] && \text{since }  \tau(\bY, \cdot \,) \text{ is a bijection} \, . 
\end{align*}

The proof of necessity is similar to the proof of \cite[Theorem 4.1]{thorisson:2007} and postponed to \Cref{sec:appendix}.

\end{proof}

\subsection{Another view on the \enquote{time-change formula}}
If $\bY=(Y_{\bi})_{\bi\in \Z^d}$ is a tail process, we define its \textit{spectral (tail) process} $\bTheta=(\Theta_{\bi})_{\bi \in \Z^d}$ by
$$\Theta_{\bi}=|Y_{\bo}|^{-1} Y_{\bi} \, , \, \bi \in \Z^d \, .$$
Observe that always $|\Theta_{\bo}|=1$. It is known that $|Y_{\bo}|$ is necessarily $\mathrm{Pareto}(\alpha)$-distributed for some $\alpha>0$, and independent of $\bTheta$. Moreover, the spectral process satisfies the following invariance property:
 \begin{equation}\label{eq:time-change0}
\ex[h\left(|\Theta_{\bk}|^{-1} \shift{\bk}\bTheta \right) \ind{\Theta_{\bk}\neq 0}]=\ex\left[h\left(\bTheta\right)|\Theta_{-\bk}|^{\alpha}\right] \, ,
\end{equation}
for all measurable $h: \R^{\Z^d}\to \R_{+}$ and all $\bk \in \Z^d$.
This property is  usually stated in the slightly different, but equivalent, form:
\begin{equation}\label{eq:time_change_original}
\ex\left[h\left(\shift{-\bk}\bTheta\right)\ind{\Theta_{-\bk}\neq 0}\right]=\ex[h\left(|\Theta_{\bk}|^{-1} \bTheta \right) |\Theta_{\bk}|^{\alpha}] \, ,
\end{equation}
for all measurable $h:\R^{\Z^d}\to \R_{+}, \bk \in \Z$; the expression under the expectation on the left-hand side is understood to be 0 if $|\Theta_{\bk}|=0$, see e.g.~\cite[Theorem 3.1(iii)]{basrak:segers:2009}. In the case $d=1$, (\ref{eq:time_change_original}) is often called the \enquote{time-change formula}.

Property  (\ref{eq:time-change0}) characterizes the class of all tail processes in the sense that if $\bTheta=(\Theta_{\bi})_{\bi \in \Z^d}$ is a process satisfying  $|\Theta_{\bo}|=1$ and (\ref{eq:time-change0}) for some $\alpha>0$, then there exists a stationary  random field which is regularly varying with index $\alpha$ and whose spectral tail process is exactly $\bTheta$. Jan{\ss}en~\cite[Theorem 4.2]{janssen:2019} and Dombry et al.~\cite[Theorems 2.9 and 3.7]{dombry:hashorva:soulier:2018} (independently)  proved this  for the case $d=1$, and Hashorva~\cite[Theorem 2.3]{hashorva:2020} extended it to the case $d\geq 2$.

We give an another interpretation of the property (\ref{eq:time-change0}): under the spectral decomposition, it is equivalent to the exceedance-stationarity of the tail process. 

\begin{proposition}\label{prop:TCF_another_view}
Let $\bTheta=(\Theta_{\bi})_{\bi\in \Z^d}$ be a random element of $\R^{\Z^d}$ such that $|\Theta_{\bo}|=1$ a.s.\ and $Y$ a $\mathrm{Pareto}(\alpha)$-distributed random variable for some $\alpha>0$, and assume that $\bTheta$ and $Y$ are independent. Then (\ref{eq:time-change0}) holds for all measurable $h: \R^{\Z^d}\to \R_{+}$ and all $\bk \in \Z^d$ if and only if  $\bY:=Y\cdot \bTheta$ is exceedance-stationary.
\end{proposition}

\begin{corollary}\label{cor:tail_pr_charact}
Let $\bY=(Y_{\bi})_{\bi\in \Z^d}$ be a random element of $\R^{\Z^d}_0$. Then $\bY$ is the tail process of some stationary regularly varying random field with tail index $\alpha>0$ if and only if $\bY$ is exceedance-stationary and $|Y_{\bo}|$ is $\mathrm{Pareto}(\alpha)$-distributed and independent of $\bTheta:=\bY/|Y_{\bo}|$.
\end{corollary}

\begin{proof}[Proof of \Cref{prop:TCF_another_view}]
If (\ref{eq:time-change0}) holds, by the previous discussion $\bY$ is necessarily the tail process of some stationary regularly varying field and thus exceedance-stationary. Hence, it only remains to prove sufficiency.

Assume that $\bY=Y\cdot \bTheta$ is exceedance-stationary, or equivalently, that  (\ref{eq:time-change1}) holds for all $h: \R^{\Z^d}\to \R_{+}$ and all $\bk \in \Z^d$. 
We will show that this implies 
\begin{align}\label{eq:TCF_inter1}
\ex[h(\shift{\bk}\bY)\ind{|Y_{\bk}|>s}]=s^{-\alpha} \ex[h(s\bY)\ind{|Y_{-\bk}|>1/s}] \, , \,
\end{align}
for all $s\in (0,1]$ and all measurable $h: \R^{\Z^d}\to \R_{+}$ and all $\bk \in \Z^d$. Indeed, 
\begin{align*}
\ex[h(\shift{\bk}\bY)\ind{|Y_{\bk}|>s}] &= \ex\Big[ \int_{1}^{\infty} h(y \shift{\bk}\bTheta)\ind{y|\Theta_{\bk}|>s} \alpha y ^{-\alpha-1} \mathrm{d}y \Big] && \\
&=s^{-\alpha} \ex\Big[ \int_{1/s}^{\infty} h(sz \shift{\bk}\bTheta)\ind{z|\Theta_{\bk}|>1} \alpha z ^{-\alpha-1} \mathrm{d}z \Big] && (z=y/s) \\
&= s^{-\alpha} \ex[h(s\shift{\bk}\bY)\ind{|Y_{\bk}|>1, |Y_{0}|>1/s}] && (1/s \geq 1)\\
&= s^{-\alpha} \ex[h(s\bY)\ind{|Y_{-\bk}|>1, |Y_{-\bk}|>1/s}] && \text{use } (\ref{eq:time-change1}) \\
&= s^{-\alpha} \ex[h(s\bY)\ind{|Y_{-\bk}|>1/s}] \, . && 
\end{align*}

Next, observe that by monotone convergence \begin{align*}
\ex[h\left(|\Theta_{\bk}|^{-1} \shift{\bk}\bTheta \right) \ind{\Theta_{\bk}\neq 0}] =\lim_{s\to 0} \ex[\tilde{h}\left(\shift{\bk}\bY \right) \ind{|Y_{\bk}|> s}] \, , 
\end{align*}
where $\tilde{h}(\bx)=h(|x_{\bo}|^{-1} \bx)$, $\bx\in \R^{\Z^d}_0$. Now applying (\ref{eq:TCF_inter1}) to $\tilde{h}$ and using the spectral decomposition of $\bY$ one can easily show that (\ref{eq:time-change0}) holds;
the calculation is the same as in the proof of \cite[Theorem 5.3.1]{kulik:soulier:2020}, but for the convenience of the reader given in \Cref{sec:appendix}.  

\end{proof}

\begin{remark}
As in \Cref{rem:mecke_tail}, it is easily shown that (\ref{eq:time-change0}) for all measurable $h:\R^{\Z^d}\to \R_+$ is equivalent to
\begin{align*}
\ex\Big[\sum_{\bk \in \Z^d} g\left(-\bk, |\Theta_{\bk}|^{-1} \shift{\bk}\bTheta \right) \ind{\Theta_{\bk}\neq 0} \Big] = \ex\Big[\sum_{\bk \in \Z^d} g\left(\bk, \bTheta\right)|\Theta_{\bk}|^{\alpha}\Big] \, , 
\end{align*}
for all measurable $g:\Z^d \times \R^{\Z^d}\to \R_+$, which can be seen as the spectral tail process analogue of the \enquote{Mecke's integral equation}.
\end{remark}

\section{Palm dualities}\label{subs:anchor}

Let $\bY=(Y_{\bi})_{\bi\in\Z^d}$ be the tail process of an $\R$-valued, stationary and regularly varying random field $\bX=(X_{\bi})_{\bi \in \Z^d}$. In this section we restrict to the case when $\bY$ a.s.\ has finitely many exceedance-points, that is $\bi$'s such that $|Y_{\bi}|>1$. We start by defining the anchored tail process introduced in \cite{basrak:planinic:2020} which, under an appropriate condition, can be seen as the {typical} cluster of extremes of the underlying random field $\bX$, see \Cref{sub:randomized_origin}, and appears in limit theorems which concern the extremal behavior of $\bX$, see e.g.~\cite[Theorem 3.9]{basrak:planinic:2020}.  Using the exceedance-stationarity property of $\bY$, that is property (\ref{eq:time-change1}), we then establish certain Palm-like dualities between the (spectral) tail process and the anchored (spectral) tail process, and use them to give a new characterization of the class of so-called representatives of the anchored (spectral) tail process. Moreover, using these dualities and the concept of exceedance-stationarity we give a new characterization of the class of tail processes with finitely many exceedance-points. In the end of the section, we give a formal meaning to the intuitive idea of choosing an (asymptotic) exceedance-point of $\bX$ \enquote{uniformly at random}.


\subsection{Anchored tail process}
\label{sub:AnchoredTailProcess}
Recall that $\e(\bx)$, for $\bx\in \R^{\Z^d}$, denotes the set of all exceedance-points of $\bx$, and that  $\espace$
consists of all elements of $\R_0^{\Z^d}$ having finitely many exceedance-points.
Further, call a function $h$ on $B\subseteq \R^{\Z^d}$ \textit{shift-invariant} on $B$ if $h(\shift{\bk}\bx)=h(\bx)$ for all $\bx\in B$ and all $\bk\in \Z^d$.

We will say that a measurable function $A:\espace \to \Z^d$ is an \textit{anchoring function} if it is a {shift-covariant} exceedance-map, that is, for all $\bx\in \espace$
\begin{enumerate}[(i)]
\item $A(\bx)\in \e(\bx)$; 
\item $A(\shift{\bk}\bx)+\bk=A(\bx)$ for all $\bk\in \e(\bx)$,
\end{enumerate} 
where, recall, the shift operator $\shift{\bk}$, $\bk \in \Z^d$, defined in (\ref{eq:shift_operator}) moves the origin to $\bk$.

We will call $A(\bx)$ the \textit{anchoring exceedance-point} or simply the \textit{anchor} of $\bx$ (with respect to $A$). 
Standard examples of anchoring functions are 
\begin{itemize}
\renewcommand\labelitemi{--}
\item \textit{first exceedance}: $A^{fe}(\bx)=\min \e(\bx)$,
\item \textit{first maximum}: $A^{fm}(\bx)=\min \{\bk\in \Z^d : |x_{\bk}|=\max_{\bi \in \Z^d} |x_{\bi}|\}$, 
\end{itemize}
where $\min$ is taken with respect to an arbitrary, but fixed, group order on $\Z^d$.

Since for any anchoring map $A$ the associated map $A(\bx,\cdot\,)$ from (\ref{eq:ass_exc_map}) is constant on $\e(\bx)$ (equal to $A(\bx)$) for all $\bx\in \espace$, anchoring maps can be seen as the opposite of bijective exceedance-maps. Moreover, the induced shift $\shift{A}$ from (\ref{eq:exc_shift}), which moves the origin to the anchor, satisfies
\begin{align}\label{eq:SI_of_anchorShift}
\shift{A} \shift{\bk}\bx = \shift{A} \bx \, , \; \bx\in \espace, \, \bk\in \e(\bx) \, . 
\end{align}

Note that any anchoring function $A$ (defined as a function on $\espace$) can and will be extended to the space 
$$\mathbb{E}=\{\bx\in \R^{\Z^d} : 1\leq |\e(\bx)| <\infty\}$$ 
by setting $A(\bx):=A(\shift{\bk}\bx)+\bk$ for $\bx\in \mathbb{E}\setminus \espace$ where $\bk$ is an arbitrary exceedance-point of $\bx$. \tcb{Due to (ii) above, this does not depend on the choice of $\bk$. Moreover, it is easily checked that then $A:\mathbb{E}\to\Z^d$ satisfies properties (i),(ii) and (\ref{eq:SI_of_anchorShift}) above for all $\bx\in \mathbb{E}$ and all $\bk\in \Z^d$.} In particular, $\shift{A}$ is shift-invariant on $\mathbb{E}$.

\begin{definition}
Let $\bY$ be the tail process of a stationary random field such that $\pr(\bY\in \espace)=1$, and $A$ an arbitrary anchoring function. Denote by $\bZ^A=(Z_{\bi}^A)_{\bi \in \Z^d}$ a random element of $\espace$ with distribution
\begin{align*}
\pr(\bZ^A\in \cdot \, ) = \pr(\bY\in \cdot\, \mid A(\bY)=\bo) \, ,
\end{align*}
and call it the \textbf{anchored process} of $\bY$ (with respect to $A$) or simply the {anchored tail process}. 
\end{definition}
This is always well-defined since $\vartheta_A:=\pr(A(\bY)=0)>0$ for every $A$ as a consequence of exceedance-stationarity of the tail process; see \cite[Lemma 3.4]{basrak:planinic:2020}. 

As in the informal discussion preceding \Cref{exa:MA_intro}, assume that high threshold exceedances of the underlying random field $\bX$ group into i.i.d.\ clusters, and think of the tail process $\bY$ as the cluster which contains and is centered around the uniformly chosen exceedance-point of $\bX$. Conditioning $\bY$ on having the anchor at the origin, i.e.\ that $A(\bY)=\bo$, corresponds to choosing uniformly only among the exceedance-points which are anchors of their respective cluster. Since, for every fixed anchoring function $A$, each cluster contains exactly one exceedance-point which is its anchor, i.e.\ for all $\bx\in \espace$,
\begin{align*}
\sum_{\bk\in \e(\bx)} \ind{A(\shift{\bk}\bx)=\bo} = \sum_{\bk\in \e(\bx)} \ind{A(\bx)=\bk} = 1 \, ,
\end{align*}
this is (up to a shift) equivalent to choosing one of the clusters uniformly at random. This heuristically explains why the anchored tail process represents (up to a shift) the distribution of a typical cluster of extremes of $\bX$, see \Cref{subs:randomly_choosing} below. Similar reasoning yields that $\vartheta_A^{-1}=\pr(A(\bY)=0)^{-1}$ should equal the expected number of exceedance-points in a typical cluster of extremes (i.e.\ that (\ref{eq:theta_is_reciprocal}) below holds), and consequently why $\vartheta$ turns out to be the extremal index of the field $\bX$, see~\cite[Remark 3.11]{basrak:planinic:2020} for the precise statement.

\subsection{Palm duality for the tail process}\label{subs:palm_tail}
We are now ready to prove the basic duality between the tail and the anchored tail process. Note again that we prove our results without any assumption on the dependence of the underlying regularly varying random field -- we only assume that its tail process a.s.\ has  finitely many exceedance-points. 

\begin{proposition}\label{prop:Palm1}
Assume $\bY$ is a tail process such that $\pr(\bY\in \espace)=1$ and let $A$ be an arbitrary anchoring function. Then for every measurable $h:\espace \to [0,\infty)$,
\begin{align}\label{eq:palm_basic}
\ex[h(\bY)]= \vartheta_A \tsum_{\bk \in \Z^d} \ex\left[h(\shift{\bk}\bZ^A)\ind{|Z^A_{\bk}|>1}\right] = \vartheta_A \ex\left[ \tsum_{\bk \in \e(\bZ^A)} h(\shift{\bk}\bZ^A)\right] \, ,
\end{align}
and
\begin{align}\label{eq:palm_variant2}
\ex\left[h(\bZ^A)\right]= \vartheta_A^{-1} \ex\left[\frac{h(\shift{A}\bY)}{\abs{\e(\bY)}}\right] \, .
\end{align}
\end{proposition}

Since $\vartheta_A>0$, taking $h\equiv 1$ in (\ref{eq:palm_basic}) yields
\begin{align}\label{eq:theta_is_reciprocal}
\ex\abs{\e(\bZ^A)}=\vartheta_A^{-1} <\infty  \, .
\end{align}
Thus, if $\pr^*$ denotes the tilted probability 
$$\pr^*(B)=\vartheta_A \ex\Big[ \abs{\e(\bZ^A)} \1{B} \Big] \, , $$ 
we can rewrite (\ref{eq:palm_basic}) as
\begin{align*}
\pr(\bY\in \cdot \, ) = \pr^*(\shift{U}\bZ^A \in \cdot \,) \, ,  
\end{align*}
where $U=U(\bZ^A)$ is uniform on $\e(\bZ^A)$. Thus, the tail process can be obtained from the anchored process by biasing its distribution by the total number of the exceedance-points and then choosing the origin uniformly at random among the set of all exceedance-points. 
Conversely, (\ref{eq:palm_variant2}) implies that, instead by conditioning, the anchored tail process can be obtained from the tail process by debiasing its distribution by the total number of the exceedance-points and then shifting the origin to the anchor.


Relations (\ref{eq:palm_basic}) and (\ref{eq:palm_variant2}) can be seen as a version of the \textit{point-at-zero} (in our case, \textit{anchor-at-zero}) Palm duality as described in \cite[Chapters 8 and 9]{thorisson:2000}. \tcb{Here, one should think of anchors (that is, exceedance-points which are anchors of their clusters) as \enquote{points} in the  index set of all exceedance-points.} The tail process $\bY$ then plays the role of the stationary (that is, exceedance-stationary) process,  while the anchored process $\bZ^A$ represents its  {Palm version} -- it is the tail process conditioned on having an anchor at the origin; \tcb{observe that this is in contrast to the case in  \Cref{rem:tail_is_Palm}.}  
Compared to standard Palm theory, relation (\ref{eq:palm_basic}) represents the \textit{inversion formula}, cf.~\cite[Formula (9.17)]{last:penrose:2017}.\footnote{\tcb{It is not clear how one could state a corresponding Campbell theorem since the tail process, as a single cluster, always contains precisely one anchor (that is, point). One would need to consider an infinite number of clusters and maybe use or modify the framework of Sigman and Whitt~\cite{sigman:whitt:2019}, but we will not pursue this approach here.}}
We note that formulas (\ref{eq:palm_basic}) and (\ref{eq:palm_variant2}) (for the case $d=1$ and $A=A^{fe}$) appear independently in \cite[Excercise 5.29]{kulik:soulier:2020}.

\begin{proof}[Proof of \Cref{prop:Palm1}]
Using the shift-covariance of $A$ and the exceedance-stationarity of $\bY$, that is property (\ref{eq:time-change1}), we get
\begin{align*}
\vartheta_A \sum_{\bk \in \Z^d}\ex\left[h(\shift{\bk}\bZ^A)\ind{|Z_{\bk}^A|>1}\right] & = \sum_{\bk \in \Z^d} \ex[h(\shift{\bk}\bY)\ind{|Y_{\bk}|>1, A(\bY)=0}] \\
& = \sum_{\bk \in \Z^d} \ex[h(\shift{\bk}\bY)\ind{|Y_{\bk}|>1, A(\shift{\bk}\bY)=-\bk}] \\
&=\sum_{\bk \in \Z^d} \ex[h(\bY)\ind{|Y_{-\bk}|>1, A(\bY)=-\bk}] \\
& = \sum_{\bk \in \Z^d} \ex[h(\bY)\ind{A(\bY)=-\bk}] = \ex[h(\bY)] \, .
\end{align*}
The last equality in (\ref{eq:palm_basic}) follows from Fubini's theorem and the definition of $\e(\bZ^A)$. 

Equality in (\ref{eq:palm_variant2}) follows by applying (\ref{eq:palm_basic}) to the function
\begin{align*}
\tilde{h}(\bx) = h(\shift{A}\bx) / \abs{\e(\bx)} \, , \, \bx\in \espace.
\end{align*}
Indeed, by (\ref{eq:SI_of_anchorShift}) the functions $\tilde{h}$ is shift-invariant, and since $\shift{A}\bZ^A=\bZ^A$ by definition we have 
$$\tilde{h}(\shift{\bk}\bZ^A)=\tilde{h}(\bZ^A) = h(\bZ^A)/ \abs{\e(\bZ^A)}\, , $$ 
for all $\bk \in \Z^d$. Thus, 
\begin{align*}
\ex\left[\frac{h(\shift{A}\bY)}{\abs{\e(\bY)}}\right]&=\ex[\tilde{h}(\bY)]  = \vartheta_A \ex\left[ \tsum_{\bk \in \e(\bZ^A)} \tilde{h}(\shift{\bk}\bZ^A)\right] \\
& = \vartheta_A \ex\left[ h(\bZ^A) \cdot \abs{\e(\bZ^A)}/\abs{\e(\bZ^A)}\right] =\vartheta_A \ex\left[ h(\bZ^A) \right] \, .
\end{align*}
\end{proof}

The following corollary is an immediate consequence of (\ref{eq:palm_variant2}).

\begin{corollary}\label{cor:palm}
\begin{itemize}
\item[(i)] For an arbitrary anchoring function $A$, 
\begin{align*}
\vartheta_A = \ex\left[ \abs{\e(\bY)}^{-1}\right]=:\vartheta \, .
\end{align*}
Thus, the probability $\vartheta_A=\pr(A(\bY)=0)$ does not depend on the choice of the anchoring function $A$.
\item[(ii)] For different anchoring functions anchored tail processes are simply randomly shifted version of one another, that is,
\begin{align}\label{eq:shift_equivalence_ZA}
\shift{A'} \bZ^A \eind \bZ^{A'}
\end{align}
for any two anchoring functions $A,A'$.
\item[(iii)] For an arbitrary anchoring function $A$, the distributions of $\shift{A}\bY$ and $\bZ^A$ are equivalent, that is, $\pr(\shift{A}\bY\in B)=1$ if and only if $\pr(\bZ^A\in B)=1$ for every measurable subset $B\subseteq \espace$.
\end{itemize}
\end{corollary}

The quantity $\vartheta=\ex\left[ \abs{\e(\bY)}^{-1}\right]$ will simply be called the \textit{(candidate) extremal index} of $\bY$, see~\cite[Remark 3.11]{basrak:planinic:2020}. On the other hand, (\ref{eq:shift_equivalence_ZA}) motivates the following definition.


\begin{definition}
If $\bZ$ is a random element of $\mathbb{E}$ such that $\shift{A}\bZ \eind \bZ^A$ for some (and then for every) anchoring function $A$, $\bZ$ will be called a \textbf{representative} of the anchored process. 
\end{definition}

Since for each $h:\espace\to \R_{+}$, the mapping $h\circ \shift{A}$ is shift-invariant on $\mathbb{E}$, $\bZ$ is a representative if and only if for some (and then for every) anchoring function $A$,
\begin{align*}
\ex[h(\bZ)]=\ex[h(\bZ^A)]
\end{align*}
for all measurable $h:\mathbb{E} \to \R_+$ which are shift-invariant; this can be stated as equality in distribution in the quotient space of $\mathbb{E}$ modulo shift-equivalence but we omit the details. 

If $\bZ$ is an arbitrary representative of the anchored process, since $\bx\mapsto |\e(\bx)|$ is shift-invariant, it follows directly that 
\begin{align*}
\ex\abs{\e(\bZ)} =  \vartheta^{-1}<\infty \, . 
\end{align*}

Moreover, \Cref{prop:Palm1}  yields the following characterization of representatives of the anchored process.

\begin{proposition}\label{prop:Palm_representatives}
Let $\bY$ be the tail process of a stationary random field such that $\pr(\bY\in \espace)=1$, and $\bZ=(Z_{\bi})_{\bi\in \Z^d}$ a random element of $\mathbb{E}$ such that $\ex\abs{\e(\bZ)}<\infty$. The following statements are equivalent.
\begin{itemize}
\item[(i)] $\bZ$ is a representative of the anchored process.
\item[(ii)] For all measurable $h:\espace\to \R_{+}$,
\begin{align}\label{eq:tail_palm_construction}
\ex[h(\bY)]= \frac{1}{\ex\abs{\e(\bZ)}} \ex\left[ \tsum_{\bk \in \e(\bZ)} h(\shift{\bk}\bZ)\right] \, .
\end{align}
\item[(iii)] For all measurable $h:\mathbb{E}\to \R_{+}$ which are shift-invariant,
\begin{align}\label{eq:tail_palm_variant_construction}
\ex\left[h(\bZ)\right]= \vartheta^{-1} \ex\left[\frac{h(\bY)}{\abs{\e(\bY)}}\right] \, .
\end{align}
\end{itemize}
\end{proposition}

\begin{proof}
If $\bZ$ is a representative, then (\ref{eq:tail_palm_construction}) follows from (\ref{eq:palm_basic}) since the mapping $\bz\mapsto \sum_{\bk\in\e(\bz)} h(\shift{\bk}\bz)$ is shift-invariant on $\mathbb{E}$ for all $h$. The fact that (ii) implies (iii) follows as in the proof of \Cref{prop:Palm1}. Finally, if (iii) holds, (i) follows from (\ref{eq:palm_variant2}) and the definition of a representative of the anchored process.
\end{proof}

Finally, note that the definition of the anchored process, as well as the conclusions of \Cref{prop:Palm1},  \Cref{cor:palm} and \Cref{prop:Palm_representatives} remain valid when $\bY$ is an arbitrary exceedance-stationary random element of $\espace$. We now show that this class is in fact characterized by the inversion formula (\ref{eq:tail_palm_construction}).

%

\begin{proposition}\label{prop:exc_stat_constr}
Let $\bZ$ be an arbitrary random element of $\mathbb{E}$ such that $\ex\abs{\e(\bZ)}<\infty$, and let $\bY$ be a random element of $\espace$ with distribution satisfying (\ref{eq:tail_palm_construction})
for all measurable $h:\espace\to \R_{+}$. Then $\bY$ is exceedance-stationary and $\bZ$ is necessarily a representative of its anchored process.
%
\end{proposition}


\begin{proof}
Relation (\ref{eq:tail_palm_construction}) implies that
\begin{align*}
\pr(\bY\in \cdot \, ) = \pr^*(\shift{U}\bZ\in \cdot \,) \, ,  
\end{align*}
where $U=U(\bZ)$ is uniform on $\e(\bZ)$ and $\pr^*$ the tilted probability satisfying $\pr^*(B)=\tfrac{1}{\ex\abs{\e(\bZ)}} \ex\big[\abs{\e(\bZ)} \1{B} \big]$. 
We will show that $\bY$ is exceedance-stationary by verifying (\ref{eq:point_stat}). 

Assume w.l.o.g.\ that $\bZ$ is an element $\espace$, i.e.\ that it has an exceedance-point at the origin -- if not, simply take $\shift{A}\bZ$ for an arbitrary anchoring function $A$. Let $\tau$ be an arbitrary bijective exceedance-map on $\espace$. Then
\begin{align*}
\pr(\shift{\tau} \bY \in \cdot \,) =  \pr^*(\shift{\tau}\shift{U}\bZ\in \cdot \,) = \pr^*(\shift{\tau(\shift{U} \bZ)+U}\bZ\in \cdot \,) \, .
\end{align*} 
Since $\tau$ is assumed bijective, $\bk\mapsto \tau(\shift{\bk} \bZ) + \bk=\tau(\bZ,\bk)$ is a bijection on $\e(\bZ)$ and thus the distribution of $\tau(\shift{U} \bZ)+U$ is again uniform on $\e(\bZ)$. Thus, $\pr(\shift{\tau} \bY \in \cdot \,)= \pr(\bY \in \cdot \,)$, i.e.\ (\ref{eq:point_stat}) indeed holds. The fact that $\bZ$ is a representative of the anchored process of $\bY$ follows from \Cref{prop:Palm_representatives}.
%
%
%
\end{proof}


The proof of \Cref{prop:exc_stat_constr} gives a nice intuition on why the tail process $\bY$ (with finitely many exceedance-points) is exceedance-stationary in the sense of \Cref{def:point_stat}, i.e.\ for every bijective exceedance-map $\tau$, $\shift{\tau}\bY\teind \bY$. To see this, as in the discussion preceding \Cref{exa:MA_intro} (and in the end of \Cref{sub:AnchoredTailProcess}), think of the tail process as the cluster of extremes of the underlying random field $\bX$ containing and being centered around an exceedance-point which is chosen uniformly at random. Applying $\shift{\tau}$ to this recentered cluster shifts to the origin to the exceedance-point of the cluster chosen by $\tau$. Since $\tau$ is bijective (think of $\tau_1$ from \Cref{exa:finitely_many_exc}), the chosen exceedance-point is again \enquote{uniformly distributed} among all of the exceedance-points of $\bX$.

\subsection{Palm duality for the spectral process}\label{subs:palm_spectral}
Let $\bY$ be a tail process of a stationary random field, and $\bTheta=\bY/|Y_{\bo}|$ its spectral (tail) process. Since $|Y_{\bo}|$ is $\mathrm{Pareto}(\alpha)$-distributed and independent of $\bTheta$ are, condition $\pr(\bY\in \espace)=1$ is easily seen to be equivalent to
\begin{align*}
\pr(\tlim_{|\bi| \toi} |\Theta_{\bi}| = 0)=\pr(\tlim_{|\bi| \toi} |Y_{\bi}| = 0)=1 \, ,
\end{align*} 
where  $|\bi|=\max_{k=1,\dots,d}|i_k|$ for $\bi=(i_1,\dots,i_d)\in \Z^d$. Thus, one can regard the tail and the spectral tail process as elements of the Banach space
\begin{align*}
l_0=\{\bx\in \R^{\Z^d} : \tlim_{|\bi|\to\infty}|x_{\bi}|= 0\} \, ,
\end{align*}
 with the
uniform norm 
$$\|\bx\| := \sup_{\bi\in\Z^d} |x_{\bi}| \, , \; \bx\in l_0 \, ,$$ 
since the corresponding Borel $\sigma$-algebra equals the product $\sigma$-algebra on $l_0$. \tcb{We note that the notation $\|\cdot\|$ is used instead of the usual $\|\cdot\|_{\infty}$ since this is the only norm we will consider on $l_0$.}

For any representative $\bZ$ of the anchored process define the corresponding \textit{spectral} process $\bQ$ by $\bQ=\bZ/\|\bZ\|$; note that by construction $\|\bQ\|=1$ almost surely. Moreover, for each anchoring function $A$, denote by $\bQ^A$ the spectral process of $\bZ^A$. 

\begin{definition}
Any random element $\bQ$ of $l_0$ such that 
$$\ex[h(\bQ)]=\ex[h(\bQ^A)] \, ,$$ 
holds for all shift-invariant measurable $h:l_0\to\R_+$, for at least one (and then for every) anchoring function $A$, will be called a representative of the anchored spectral (tail) process.
\end{definition}
Note that the first maximum anchoring function $A=A^{fm}$ and the corresponding shift $\shift{A}$ are well-defined on any element of $l_0\setminus\{\bo\}$, and that $A(t\bx)=A(\bx)$ for all $\bx\in l_0\setminus\{\bo\}, t>0$. In particular, the extremal index $\vartheta$ satisfies 
\begin{align*}
\vartheta=\pr(A^{fm}(\bY)=0)=\pr(A^{fm}(\bTheta)=0)\, .
\end{align*}
The following result follows from \cite[Lemma 3.7 and Remark 3.7]{basrak:planinic:2020} and shows that the spectral decomposition of the tail process carries over to the anchored process.
 
 \begin{lemma}\label{lem:polar_decomposition_Z}
Let $\bY$ be a tail process with index $\alpha>0$ and such that $\pr(\bTheta\in l_0)=1$.
\begin{itemize}
\item[(i)]  For every representative $\bZ$ of the anchored process, $\|\bZ\|$ is $\mathrm{Pareto}(\alpha)$-distributed.
\item[(ii)]  
A random element $\bQ$ in $l_0$ with distribution
\begin{align*}
\pr(\bQ \in \cdot \,)=\pr(\bTheta \in \cdot \mid A^{fm}(\bTheta)=0) \, 
\end{align*}  
is a representative of the anchored spectral process.
\item[(iii)] There exists a representative $\bZ$ of the anchored process such that $\|\bZ\|$ and $\bQ$ are independent.
\item[(iv)] If $\bQ$ is an arbitrary representative of the anchored spectral process and $Y$ is $\mathrm{Pareto}(\alpha)$-distributed and independent of $\bQ$, $\bZ:=Y\cdot \bQ$ is a representative of the anchored process.
\end{itemize}
\end{lemma}

%

We are now in position to establish a duality between the spectral tail process and representatives of the anchored spectral process.

\begin{proposition}\label{prop:exc_stat_representation}
Let $\bY$ be a tail process with index $\alpha>0$ and such that $\pr(\bTheta\in l_0)=1$. Let $\bQ=(Q_{\bi})_{\bi\in \Z^d}$ be any representative of its  anchored spectral process. Then 
\begin{align}\label{eq:palmSpectral_basic}
\ex[h(\bTheta)]= \vartheta \ex\left[\tsum_{\bk \in \Z^d}h\left(\frac{\shift{\bk}\bQ}{|Q_{\bk}|}\right)|Q_{\bk}|^{\alpha}\right] \, ,
\end{align}
and
\begin{align}\label{eq:palmSpectral_variant2}
\ex\left[h(\shift{A}\bQ)\right] = \vartheta^{-1} \ex\left[h\left(\frac{\shift{A}\bTheta}{\|\bTheta\|}\right) \cdot \frac{\|\bTheta\|^{\alpha}}{\sum_{\bk \in \Z^d}|\Theta_{\bk}|^{\alpha}}\right] \, .
\end{align} 
for all measurable $h:l_0\to \R_{+}$, where $A=A^{fm}$ is the first maximum anchoring function. The $\bk$th summand on the right hand side of (\ref{eq:palmSpectral_basic}) is understood to be zero if $|Q_{\bk}|=0$.
\end{proposition}

Relation (\ref{eq:palmSpectral_basic}) yields
\begin{align*}
\ex\left[\tsum_{\bk \in \Z^d}|Q_{\bk}|^{\alpha}\right]=\vartheta^{-1} <\infty  \, ,
\end{align*}
and
\begin{align}\label{eq:spectral_change_of_measure}
\pr(\bTheta\in \cdot \, ) = \pr^*\left(\frac{\shift{T}\bQ}{|Q_{T}|} \in \cdot \,\right) \, ,  
\end{align}
where $\pr^*$ satisfies $\pr^*(B)=\vartheta \ex\left[ \tsum_{\bk \in \Z^d}|Q_{\bk}|^{\alpha} \1{B}\right]$, and $T=T(\bQ)$ is a random element of $\Z^d$ with distribution
\begin{align*}
\pr^*(T=\bk \mid \bQ)=\frac{|Q_{\bk}|^{\alpha}}{\tsum_{\bj \in \Z^d}|Q_{\bj}|^{\alpha}} \, , \, \bk \in \Z^d.
\end{align*}

Similarly, (\ref{eq:palmSpectral_variant2}) implies 
\begin{align}\label{eq:extremal_repre_spectral}
\ex\left[ \frac{\|\bTheta\|^{\alpha}}{\sum_{\bk \in \Z^d}|\Theta_{\bk}|^{\alpha}}\right]=\vartheta>0  \, ,
\end{align}
and that the distribution of $\bQ$ (up to a random shift) can be obtained by debiasing the distribution of $\bTheta/\|\bTheta\|$ by $\tsum_{\bk \in \Z^d}|\Theta_{\bk}|^{\alpha}/\|\bTheta\|^{\alpha}$. 

For $d=1$, (\ref{eq:palmSpectral_basic}) essentially appears in \cite[Remark 2.13]{dombry:hashorva:soulier:2018} but is stated in terms of the so-called tail measure $\nu$ of $\bY$ and called a \textit{moving shift} representation of $\nu$. On the other hand, equality (\ref{eq:palmSpectral_variant2}) was known to hold only for $\alpha$-homogeneous (and shift-invariant) functions $h$, see \cite[Lemma 3.7]{planinic:soulier:2018}. In \Cref{sec:MA}  we utilize (\ref{eq:palmSpectral_variant2}) to determine the distribution of representatives of the anchored process for a large class of tail processes. 

In the terminology of \cite[Introduction]{drees:2021}, (\ref{eq:spectral_change_of_measure}) means that the distribution of $\bTheta$ is the \textit{RS-transform} of the distribution of $\bQ$ under $\pr^*$. As noted in  \cite{drees:2021}, it is easy to verify that every RS-transformed distribution is invariant under the RS-transformation. In particular, (\ref{eq:spectral_change_of_measure}) implies that the distribution of $\bTheta$ is invariant under the RS-transformation. This gives an alternative proof of this invariance property of $\bTheta$, originally noticed and proved in \cite[Theorem 2.4]{janssen:2019}.

\begin{proof}[Proof of \Cref{prop:exc_stat_representation}]

Let $Y$ be a $\mathrm{Pareto}(\alpha)$-distributed random variable, independent of $\bQ$. \Cref{lem:polar_decomposition_Z} implies that $\bZ=(Z_{\bi})_{\bi\in \Z^d}:=(Y\cdot Q_{\bi})_{\bi\in \Z^d}$ is a representative of the anchored process. \tcb{For every  measurable $g:(1,\infty)\times l_0\to \R_{+}$, applying (\ref{eq:tail_palm_construction}) to the function
\begin{align*}
h(\bx):=g(|x_{\bo}|, \bx/|x_{\bo}|)\1{\{|x_{\bo}|>1\}} \, , \; \bx\in l_0 \, ,
\end{align*}
yields}
%
\begin{align}
\ex[g(|Y_{\bo}|,\bTheta)] &=\ex[h(\bY)]=\vartheta \ex\left[\sum_{\bk\in \Z^d} g\left(|Z_{\bk}|, \frac{\shift{\bk}\bZ}{|Z_{\bk}|}\right) \ind{|Z_{\bk}|>1}\right] && \notag\\
&= \vartheta \ex\left[\sum_{\bk \in \Z^d}\int_{1}^{\infty} g\left(y|Q_{\bk}|, \frac{\shift{\bk}\bQ}{|Q_{\bk}|}\right)\ind{y|Q_{\bk}|>1}\alpha y^{-\alpha -1} dy \right] && \notag \\
&= \vartheta \ex\left[\sum_{\bk \in \Z^d}|Q_{\bk}|^{\alpha}\int_{1}^{\infty} g\left(z,\frac{\shift{\bk}\bQ}{|Q_{\bk}|}\right)\alpha z^{-\alpha -1} dz \right] &&  \text{ since } |Q_{\bk}|\le 1 \notag\\
&= \vartheta \ex\left[\sum_{\bk \in \Z^d}g\left(Y, \frac{\shift{\bk}\bQ}{|Q_{\bk}|}\right)|Q_{\bk}|^{\alpha}\right] \, . \label{eq:polar_decomposition_Z} &&
\end{align}
In particular, (\ref{eq:palmSpectral_basic}) holds. 

To prove (\ref{eq:palmSpectral_variant2}) we will apply (\ref{eq:palmSpectral_basic}) to the function
\begin{align*}
\tilde{h}(\bx)=h\left(\frac{\shift{A}\bx}{\|\bx\|}\right) \frac{\|\bx\|^{\alpha}}{\sum_{\bj \in \Z^d} |x_{\bj}|^{\alpha}} \, , \; \bx\in l_0\setminus\{\bo\} \, ,
\end{align*}
where $A=A^{fm}$. We have
\begin{align*}
\vartheta^{-1} \ex\left[h\left(\frac{\shift{A}\bTheta}{\|\bTheta\|}\right) \cdot \frac{\|\bTheta\|^{\alpha}}{\sum_{\bk \in \Z^d}|\Theta_{\bk}|^{\alpha}}\right] &= \vartheta^{-1} \ex\left[\tilde{h}(\bTheta)\right] \\
& =  \ex\left[\sum_{\bk \in \Z^d}\tilde{h}\left(\frac{\shift{\bk}\bQ}{|Q_{\bk}|}\right)|Q_{\bk}|^{\alpha}\right] \\
&= \ex\left[\sum_{\bk \in \Z^d}h\left(\frac{\shift{A}\bQ/|Q_{\bk}|}{\|\bQ\|/|Q_{\bk}|}\right)\frac{\|\bQ\|^{\alpha}/|Q_{\bk}|^{\alpha}}{\tsum_{\bj\in \Z^d}|Q_{\bj}|^{\alpha}/|Q_{\bk}|^{\alpha}}|Q_{\bk}|^{\alpha}\right] \\
&= \ex\left[h\left(\shift{A}\bQ\right)\frac{\tsum_{\bk \in \Z^d}|Q_{\bk}|^{\alpha}}{\tsum_{\bj\in \Z^d}|Q_{\bj}|^{\alpha}}\right] = \ex\left[h\left(\shift{A}\bQ\right)\right] \, .
\end{align*}
In the third equality we have used (\ref{eq:SI_of_anchorShift}) and in the fourth the fact that $\|\bQ\|=1$ almost surely. Also, note that we have silently used the fact that $\tsum_{\bk \in \Z^d}|Q_{\bk}|^{\alpha}<\infty$ a.s.\ which follows from (\ref{eq:palmSpectral_basic}).
\end{proof}

We can now state an analogue of \Cref{prop:Palm_representatives} for representatives of the anchored spectral tail process; the  proof is omitted.

\begin{proposition}\label{prop:charact_anch_spectral}
Let $\bY$ be a tail process with index $\alpha>0$ and such that $\pr(\bTheta\in l_0)=1$, and let $\bQ=(Q_{\bi})_{\bi\in \Z^d}$ be random element of $l_0$ such that $\|\bQ\|=1$ a.s.\ and $\ex\left[\tsum_{\bk \in \Z^d}|Q_{\bk}|^{\alpha}\right]<\infty$. The following statements are equivalent.
\begin{itemize}
\item[(i)] $\bQ$ is a representative of the anchored spectral tail process.
\item[(ii)] For all measurable $h:l_0\to \R_{+}$,
\begin{align}\label{eq:palmSpectral_basic_2}
\ex[h(\bTheta)]= \frac{1}{\ex\left[\tsum_{\bk \in \Z^d}|Q_{\bk}|^{\alpha}\right]} \ex\left[\tsum_{\bk \in \Z^d}h\left(\frac{\shift{\bk}\bQ}{|Q_{\bk}|}\right)|Q_{\bk}|^{\alpha}\right] \, .
\end{align}
\item[(iii)]  For all measurable $h:l_0\to \R_{+}$ which are shift-invariant,
\begin{align}\label{eq:palmSpectral_variant2_2}
\ex\left[h(\bQ)\right] = \vartheta^{-1} \ex\left[h\left(\frac{\bTheta}{\|\bTheta\|}\right) \cdot \frac{\|\bTheta\|^{\alpha}}{\sum_{\bk \in \Z^d}|\Theta_{\bk}|^{\alpha}}\right] \, .
\end{align} 
\end{itemize}
\end{proposition}

Our final result shows that exceedance-stationarity and spectral decomposition of at least one representative of the anchored process characterizes the class of tail processes with finitely many exceedance-points. This gives another, arguably more probabilistic, view on  this class of tail processes.  

\begin{proposition}\label{prop:charact}
Let $\bQ$ be an arbitrary random element of $l_0$ satisfying $\|\bQ\|=1$ a.s.\ and $\ex\left[\tsum_{\bk \in \Z^d}|Q_{\bk}|^{\alpha}\right]<\infty$ for some $\alpha>0$. Define $\bZ:=Y\cdot \bQ$ where $Y$ is a $\mathrm{Pareto}(\alpha)$-distributed random variable independent of $\bQ$. 

Then $\ex\abs{\e(\bZ)}<\infty$  and the exceedance-stationary process generated by $\bZ$, that is a random element $\bY$ of $\espace$ with distribution satisfying (\ref{eq:tail_palm_construction}), 
is the tail process of some stationary regularly varying random field with tail index $\alpha$.
\end{proposition}

\begin{remark}
One possible construction of such a stationary random field (directly from $\bQ$) can be done using \cite[Definition 2.15 and Theorem 3.7]{dombry:hashorva:soulier:2018}; we omit the details. 
\end{remark}

\begin{remark}\label{rem:charact}
The proof below gives the following variant of \Cref{prop:charact}: If $\bQ$ is an arbitrary random element of $l_0$ satisfying $\|\bQ\|=1$ a.s.\ and $\ex\left[\tsum_{\bk \in \Z^d}|Q_{\bk}|^{\alpha}\right]<\infty$ for some $\alpha>0$, the random element $\bTheta$ with distribution satisfying (\ref{eq:palmSpectral_basic_2}) is the spectral tail process of some stationary regularly varying random field with tail index $\alpha$.
\end{remark}

\begin{proof}[Proof of \Cref{prop:charact}]
Since $Y$ is a $\mathrm{Pareto}(\alpha)$-distributed and independent of $\bQ$, 
$$\ex\abs{\e(\bZ)}=\ex\left[\tsum_{\bk \in \Z^d}\ind{|Z_{\bk}|>1}\right]=\ex\left[\tsum_{\bk \in \Z^d}|Q_{\bk}|^{\alpha}\right]<\infty\, .$$
By \Cref{cor:tail_pr_charact}, it remains to show that the exceedance-stationary process $\bY$ generated by $\bZ$ admits the spectral decomposition, i.e., that $|Y_{\bo}|$ is $\mathrm{Pareto}(\alpha)$-distributed and independent of $\bTheta:=\bY/|Y_{\bo}|$. Using (\ref{eq:tail_palm_construction}) and arguing exactly as in derivation of (\ref{eq:polar_decomposition_Z}) one obtains
\begin{align*}
\ex[g(|Y_{\bo}|,\bTheta)] &=\vartheta \ex\left[\sum_{\bk\in \Z^d} g\left(|Z_{\bk}|, \frac{\shift{\bk}\bZ}{|Z_{\bk}|}\right) \ind{|Z_{\bk}|>1}\right] = \vartheta \ex\left[\sum_{\bk \in \Z^d}g\left(Y, \frac{\shift{\bk}\bQ}{|Q_{\bk}|}\right)|Q_{\bk}|^{\alpha}\right] \, ,
\end{align*} 
for every  measurable $g:(1,\infty)\times l_0\to \R_{+}$. 
Thus, $|Y_{\bo}|$ is $\mathrm{Pareto}(\alpha)$-distributed, $\bTheta$ satisfies  (\ref{eq:palmSpectral_basic_2}), and consequently $\bTheta$ and $|Y_{\bo}|$ are independent, as desired.
\end{proof}

\section{Asymptotics for clusters of extremes}\label{sub:randomized_origin}
\label{sec:randomized_origin}

\subsection{Limiting behavior of a single cluster}
Assume that $\bY=(Y_{\bi})_{\bi\in\Z^d}$ is the tail process of a stationary random field $\bX=(X_{\bi})_{\bi\in\Z^d}$ with $d\in \N$. 
\begin{proposition}\label{prop:typical_cluster}
Let $(c_n)_{n\in\N}\subseteq (0,\infty)$ and $(r_n)_{n\in\N}\subseteq \N$ be sequences  satisfying 
$c_n,r_n\to \infty$ and $r_n^d\pr(|X_{\bo}|>c_n)\to 0$, 
as $n\toi$. Assume that
\begin{align}\label{eq:AC}
\lim_{m\toi}\limsup_{n\toi}\pr\left(\max_{m\le |\bi|\leq r_n}|X_{\bi}|>c_n x\; \Big| \; |X_{\bo}|>c_n y\right)=0 \, , \, x,y>0 \, .
\end{align}
Then $\pr(\bY\in l_0)=1$ and the following convergences as $n\toi$ hold in $l_0$ with respect to the uniform norm $\|\cdot\|$:
\begin{align}\label{eq:tail_process_cluster}
\pr \big( (c_n^{-1}X_{\bi})_{\bi \in \{-r_n,\dots,r_n\}^d} \in \cdot
\,\, \big| \, |X_{\bo}|> c_n \big) \wto 
\pr \big(\bY \in \cdot \: \big) \, ,
\end{align}
and 
\begin{align}\label{eq:anchored_process_cluster}
\pr \big( \shift{A} (c_n^{-1}X_{\bi})_{\bi \in \{1,\dots,r_n\}^d} \in \cdot
\,\, \big| \, \textstyle\max_{\bi \in \{1,\dots,r_n\}^d}|X_{\bi}|> c_n \big) \wto 
\pr \big(\bZ^A \in \cdot \: \big) \, ,
\end{align}
for every anchoring function $A:l_0\to \Z^d$ such that $\pr(\bY\in \mathrm{disc}\, A)=0$ where $\mathrm{disc}\, A$ is the set of all discontinuity points of $A$.
\end{proposition}

\begin{remark}

\begin{itemize}
\item[(i)] Finite blocks of $\bX$ are embedded in $l_0$ by adding infinitely many zeros on the positions around the hypercubes $\{-r_n,\dots,r_n\}^d$ and $\{1,\dots,r_n\}^d$.
\item[(ii)] Due to the spectral decomposition of the tail process, $\pr(\bY\in \mathrm{disc}\, A)=0$ always holds when $A=A^{fe}$ is the first exceedance anchoring function.
\item[(iii)] Condition (\ref{eq:AC}) is usually called the \textit{anticlustering} or \textit{finite mean cluster size} condition.
\end{itemize}
\end{remark}

The proofs of $\pr(\bY\in l_0)=1$ and (\ref{eq:tail_process_cluster}) can be found in \cite[Theorem 6.1.4]{kulik:soulier:2020}; even though only the case $d=1$ is considered, the proof directly extends to general $d\in \N$. The convergence in (\ref{eq:anchored_process_cluster}) is a slight extension of~\cite[Equation (3.15)]{basrak:planinic:2020} which essentially says that, for an arbitrary anchoring function $A$,  convergence in (\ref{eq:anchored_process_cluster}) holds for shift-invariant subsets of $l_0$ -- this is formulated as convergence in the quotient space of $l_0$ modulo shift-equivalence; the proof of (\ref{eq:anchored_process_cluster}) is given in \Cref{sec:appendix}. 

Let us now, at least informally, identify a \textit{cluster of extremes} of $\bX$ with the block $(X_{\bi})_{\bi\in I_n}$ of observations containing at least one exceedance-point -- that is, at least one $\bi$ such that $|X_{\bi}|>c_n$, where $I_n\subseteq \Z^d$ are hypercubes with sides proportional to $r_n$, cf.~\cite[Section 6.1]{kulik:soulier:2020}. By (\ref{eq:tail_process_cluster}), the distribution of the  tail process can then be seen as the asymptotic distribution of the (rescaled) cluster $(X_{\bi})_{\bi\in I_n}$  which contains, and is centered around, a typical exceedance-point of $\bX$. 
\tcb{On the other hand,  under the terminology introduced in \Cref{sec:intro}, because of (\ref{eq:anchored_process_cluster}) we refer to the distribution of the anchored tail process as the  asymptotic distribution of  a typical cluster of extremes of $\bX$.}


\subsection{\tcb{Poisson approximation for clusters of extremes}}

For notational simplicity assume from now on that $d=1$. Divide $X_1,\dots, X_n$ into disjoint blocks
\begin{align*}
\bX_{n,j}=(X_{(j-1)r_n +1}, \dots, X_{jr_n} ) \, , \, j\in \{1,\dots, k_n\} \, , \, 
\end{align*}
of size $r_n$, where $k_n=\lfloor \tfrac{n}{r_n}\rfloor$ is the total number of blocks.  Assume that $c_n,r_n\to \infty$, $n\pr(|X_0|>c_n)\to 1$, and $r_n\pr(|X_{0}|>c_n)\sim r_n/n\to 0$. Observe that by regular variation of $|X_0|$, for every $\epsilon>0$,
\begin{align*}
\ex\Big[\sum_{i=1}^n \1{\{|X_i|> c_n\epsilon\}}\Big] = n\pr(|X_0|>c_n\epsilon) \to \epsilon^{-\alpha}:= \tau_{\epsilon} \, ,
\end{align*}
that is, the asymptotic expected number of exceedance-points of $X_1,\dots, X_n$ (over $c_n\epsilon$) is $\tau_\epsilon\in (0,\infty)$. 

Finally, in addition to (\ref{eq:AC}) assume the mixing condition on the blocks $\bX_{n,j}$ , $j=1,\dots, k_n$ given in \cite[Theorem 3.9]{basrak:planinic:2020} -- intuitively, for every fixed $\epsilon>0$, blocks containing at least one exceedance over $c_n \epsilon$ (i.e.,  clusters of extremes) asymptotically behave as if they were independent.

Then \cite[Theorem 3.9]{basrak:planinic:2020} implies that for every $\epsilon>0$ and every anchoring function $A$ such that $\pr(\bY\in \mathrm{disc}\, A)=0$, as $n\toi$,
\begin{align}\label{eq:Poisson_approx_for_clusters}
\sum_{j=1}^{k_n} \delta_{\shift{A}(\bX_{n,j}/c_n\epsilon)} \1{\{\|\bX_{n,j}\|>c_n\epsilon\}} \dto \sum_{j= 1}^{K_\epsilon} \delta_{\bZ^A_j} \, ,
\end{align}
in the space of finite point measures on the Polish space $(l_0,\|\cdot\|$)\footnote{The topology is the weak topology, while $\delta_x$ denotes the Dirac measure concentrated at $x\in l_0$.}, where 
\begin{itemize}
\item[--] $K_\epsilon$ is Poisson distributed with parameter $\vartheta\tau_\epsilon$, and
\item[--] $(\bZ_j^A)_{j\in \N}$ an iid sequence of random elements in $l_0$ distributed as $\bZ^A$, and independent of $K_\epsilon$.
\end{itemize}
Thus, high-level exceedances of $X_1,\dots, X_n$ asymptotically group into Poisson number of i.i.d.\ clusters with common distribution equal to the distribution of the anchored tail process (i.e., the typical cluster).

We note that convergence (\ref{eq:Poisson_approx_for_clusters}) follows from \cite[Theorem 3.9]{basrak:planinic:2020} (and \Cref{lem:polar_decomposition_Z}(iv)) similarly as (\ref{eq:anchored_process_cluster}) follows from \cite[Equation (3.15)]{basrak:planinic:2020}, so we will omit the details.
\subsection{\tcb{Tail process as the cluster containing a randomly chosen exceedance-point}}\label{subs:randomly_choosing}

Assume now that $n\pr(|X_{0}|>c_n)\to \infty$ and $r_n\pr(|X_{0}|>c_n)\to 0$.
Further, let conditionally on $X_1,\dots, X_n$,
\begin{enumerate}
\item $K_n$ be uniform on the set $N_n^c:=\{j\in \{1,\dots,k_n\} : \|\bX_{n,j}\|>c_n\}$; thus, $\bX_{n,K_n}$ is  a uniformly (or randomly) chosen cluster of extremes of $X_1,\dots,X_n$.
\item $T_n$ be uniform on the set $N_n^{e}:=\{i\in \{1,\dots,n\} : |X_{i}|>c_n\}$; thus, if $\bX_n^{T_n}:=\bX_{n,j}$ where $\bX_{n,j}$ is the block which contains $X_{T_n}$, $\bX_n^{T_n}$ represents the cluster containing a uniformly chosen exceedance-point of  $X_1,\dots,X_n$.\footnote{Observe that $\bX_n^{T_n}$ is not well-defined if  $T_n\in \{k_n r_n +1,\dots,n \}$. However, since 
$$\pr(T_n\in \{k_n r_n +1,\dots,n \}) \leq r_n\pr(|X_0|>c_n)\to 0\, , \; \;  \text{as } n\toi \, , $$ 
we can and will neglect this edge effect.} 
\end{enumerate}

Finally, in addition to (\ref{eq:AC}) again assume that the blocks $\bX_{n,j}$ , $j=1,\dots, k_n$ satisfy a suitable mixing condition, namely the one given in \cite[Condition (10.1.2)]{kulik:soulier:2020}:
\begin{align}\label{eq:empirical_conv_mixing_cond}
\ex\big[e^{-\frac{1}{n\pr(|X_0|>c_n)}\sum_{j=1}^{k_n} h(c_n^{-1}\bX_{n,j} )}\big] - \ex\big[e^{-\frac{1}{n\pr(|X_0|>c_n)} h(c_n^{-1}\bX_{n,1} )}\big]^{k_n} \to 0 \, , \;\; \text{ as } \; n\toi \, ,
\end{align}
for every $h:l_0\to \R_{+}$ which is bounded, shift-invariant and Lipschitz continuous, and with support being a subset of $\{\bx\in l_0 : \|\bx\|>\epsilon\}$ for some $\epsilon>0$. 

\begin{proposition}\label{prop:randomized_origin_cluster}
Under the above assumptions, as $n\toi$,
\begin{align}\label{eq:number_of_clusters_to_infty}
\frac{|N_n^c|}{\vartheta n\pr(|X_{0}|>c_n)} \Pto 1 \, ,
\end{align}
and 
\begin{align}\label{eq:number_of_exc_to_infty}
\frac{|N_n^e|}{ n\pr(|X_{0}|>c_n)} \Pto 1 \, ,
\end{align}
so in particular  $|N_n^c|, |N_n^e|\Pto +\infty$.\footnote{\enquote{$\Pto$} denotes convergence in probability.}

Moreover, for every anchoring function $A$ such that $\pr(\bY\in \mathrm{disc}\, A)=0$,
\begin{align}\label{eq:randomized_cluster}
\pr(\shift{A}(c_n^{-1} \bX_{n, K_n}) \in \cdot\, ) \wto \pr(\bZ^A \in \cdot \,) \, , 
\end{align}
and
\begin{align}\label{eq:randomized_origin}
\pr(c_n^{-1} \shift{T_n}\bX_{n}^{T_n} \in \cdot\, ) \wto 
\pr(\bY\in \cdot \,) \, , 
\end{align} 
in $l_0$ with respect to the uniform norm $\|\cdot\|$.
\end{proposition}

Results in (\ref{eq:randomized_cluster}) and (\ref{eq:randomized_origin}) give formal meaning to the statements:
\begin{enumerate}
\item The asymptotic distribution of a typical cluster of extremes of $\bX$ can be obtained by randomly choosing a cluster of extremes of $\bX$.
\item The distribution of the tail process can be seen as the asymptotic distribution of the cluster containing and being centered around a randomly chosen exceedance-point of $\bX$. 
\end{enumerate}

The proof of \Cref{prop:randomized_origin_cluster} is postponed to \Cref{sec:appendix}; it relies on the convergence of empirical cluster process given in \cite[Lemma 10.1.1]{kulik:soulier:2020}.


\begin{remark}\label{rem:mdep_mixing}
If for some $m\in \N$, the stationary regularly varying time series $\bX=(X_{i})_{i\in \Z}$ is $m$-dependent (i.e.\ $\sigma(X_i : i \leq 0)$ and $\sigma(X_i : i\geq m)$ are independent), conditions (\ref{eq:AC}) and \cite[Condition (10.1.2)]{kulik:soulier:2020} are satisfied for all sequences $(c_n)$ and $(r_n)$ such that $c_n,r_n\to \infty$, $\lim_{n\toi} n\pr(|X_{0}|>c_n)=\infty$, and $r_n\pr(|X_{0}|>c_n)\to 0$, see \cite[Lemma 6.1.3 and Theorem 10.1.3]{kulik:soulier:2020}. In particular, (\ref{eq:number_of_clusters_to_infty})-(\ref{eq:randomized_origin}) always hold. See \Cref{exa:MA} below.
\demo
\end{remark}

Finally, note that  (\ref{eq:randomized_origin})  does not necessarily hold if the mixing assumption is not satisfied. This is illustrated on a very simple model in \Cref{exa:MA_counter} below.  

\section{Moving averages with random coefficients and regularly varying innovations}\label{sec:MA}

Let $(Z_{\bi})_{\bi\in \Z^d}$ be an $\R$-valued random field of i.i.d.\ regularly varying random variables, that is, for some $\alpha>0$ and $p\in [0,1]$,  
\begin{align}\label{eq:Z_RV}
\lim_{u\toi}\frac{\pr(|Z_{\bo}|>uy)}{\pr(|Z_{\bo}|>u)}=y^{-\alpha} \, , \, y>0 \, \quad	\text{ and } \quad
\lim_{u\toi}\frac{\pr(Z_{\bo}>u)}{\pr(|Z_{\bo}|>u)} = p \, .
\end{align}
If $\alpha>1$, assume that $\ex[Z_{\bo}]=0$. 

Further, let for each $\bi\in \Z^d$, $\bC^{(\bi)}=(C_{\bi,\bk})_{\bk\in \Z^d}$ be an $\R$-valued random field over $\Z^d$, and assume that 
\begin{align}\label{eq:C_stat}
(\bC^{(\bi+\bk)})_{\bi\in \Z^d} \eind (\bC^{(\bi)})_{\bi\in \Z^d} \, , \, \text{ for all } \bk\in \Z^d \, .
\end{align}
To avoid the trivial case, we assume that $\pr(C_{\bo,\bk}=0 \text{  for all } \bk\in \Z^d)=0$.

Define the \textit{moving average} process $\bX=(X_{\bi})_{\bi\in\Z^d}$ as the (stationary) random field given by
\begin{align}\label{eq:MA_process}
X_{\bi}=\sum_{\bk\in \Z^d} C_{\bi,\bk} Z_{\bi-\bk} \, , \, \bi \in \Z^d \, .
\end{align}
\tcb{If we assume that $(\bC^{(\bi)})_{\bi\in \Z^d}$ and $(Z_{\bi})_{\bi\in \Z^d}$ are independent, and that the field $\bC^{(\bo)}=(C_{\bo,\bk})_{\bk\in \Z^d}$ satisfies the following moment condition: there exists $\epsilon\in (0,\alpha)$ such that
\begin{align*}
\tsum_{\bk \in \Z^d} \ex \big[ |C_{\bo,\bk}|^{\alpha-\epsilon} + C_{\bo,\bk}|^{\alpha+\epsilon} \big] <\infty  \,, \quad & \text{if } \alpha\in (0,1)\cup (1,2)  , \\
\ex \Big[ \big(\tsum_{\bk \in \Z^d}  |C_{\bo,\bk}|^{\alpha-\epsilon} \big)^{\frac{\alpha+\epsilon}{\alpha-\epsilon}} \Big] <\infty \,, \quad & \text{if } \alpha\in \{1,2\}  , \\
\ex \Big[ \big(\tsum_{\bk \in \Z^d}  |C_{\bo,\bk}|^{\alpha-\epsilon} \big)^{\frac{\alpha+\epsilon}{\alpha-\epsilon}} \Big] <\infty \,, \quad & \text{if } \alpha\in (2,\infty)  ,
\end{align*}
then (i) the series in (\ref{eq:MA_process}) converges a.s.\, and (ii) $X_{\bo}$ is again regularly varying with the same tail index $\alpha$; see Hult and Samorodnitsky~\cite[Theorem 3.1]{hult:samorodnitsky:2008}.}

\begin{remark}\label{rem:moment_assump_mdep}
If $\pr(C_{\bo,\bk}=0 \text{  for all } |\bk|\geq m)=1$ for some $m\in \N$ (as is the case in \Cref{exa:MA} below), the above moment condition is equivalent to the existence of $\epsilon>0$ such that
\begin{align*}
\sum_{|\bk|<m} \ex \big[ |C_{\bo,\bk}|^{\alpha+\epsilon} \big] <\infty \, .
\end{align*}
Moreover, in this case one can also allow that $\ex[Z_{\bo}]\neq 0$ if $\alpha>1$, see \cite[Remark 3.2]{hult:samorodnitsky:2008}.
\demo
\end{remark}

The above moment condition in particular ensures that 
\begin{align*}
c:=\sum_{\bk\in \Z^d} \ex \big[|C_{\bo,\bk}|^{\alpha}\big] <\infty \, .
\end{align*}
Moreover, by \cite[Theorem 15.1.2]{kulik:soulier:2020} the random field $\bX$ admits a tail process $\bY$ which has tail index $\alpha$ and whose spectral tail process $\bTheta$ satisfies
\begin{align}\label{eq:MA_spectral}
\ex[h(\bTheta)]=\frac{1}{c}\sum_{\bk\in \Z^d}\ex\Big[h\Big(\frac{(C_{\bi,\bi+\bk})_{\bi\in \Z^d} \cdot \kappa}{|C_{\bo,\bk}|}\Big)\cdot |C_{\bo,\bk}|^{\alpha} \Big] \, ,
\end{align}
for all measurable $h:\R^{\Z^d}\to\R_{+}$, where $\kappa$ is a $\{-1,1\}$-valued random variable independent of everything and satisfying $\pr(\kappa=1)=p$ with $p$ the same as in (\ref{eq:Z_RV}); \cite[Theorem 15.1.2]{kulik:soulier:2020} is proved for the case $d=1$ but its extension to the case $d\in \N$ is straightforward. Note that the $\bk$th summand on the right hand side above is understood to be 0 if $C_{\bo,\bk}=0$.

In order to to understand the extremal behavior of $\bX$ (i.e.\ determine the distribution of the (anchored) tail process) the following observation is crucial. 
By (\ref{eq:C_stat}), 
\begin{align*}
(C_{\bi,\bi+\bk})_{\bi\in \Z^d} \eind (C_{\bi+\bk,\bi+\bk})_{\bi\in \Z^d} \, , \, \text{ for all } \bk\in \Z^d \, .
\end{align*}
Thus, if $\bC=(C_{\bi})_{\bi\in\Z^d}$ is the random field of \enquote{diagonal} coefficients
\begin{align}\label{eq:C_new}
C_{\bi}=C_{\bi,\bi} \, , \, \bi \in \Z^d \, ,
\end{align}
one can rephrase (\ref{eq:MA_spectral}) as
\begin{align}\label{eq:MA_spectral_2}
\ex[h(\bTheta)]=\frac{1}{c}\ex\Big[ \tsum_{\bk\in \Z^d}h\Big(\frac{\shift{\bk}\bC \cdot \kappa}{|C_{\bk}|}\Big)\cdot |C_{\bk}|^{\alpha} \Big] \, ,
\end{align}
for all measurable $h:\R^{\Z^d}\to\R_{+}$. In particular, 
\begin{align}\label{eq:moment_assump_cor}
\sum_{\bk\in \Z^d} \ex\big[ |C_{\bk}|^{\alpha}\big]=c=\sum_{\bk\in \Z^d}\ex \big[|C_{\bo,\bk}|^{\alpha}\big] <\infty  \, ,
\end{align}
and thus $\pr(\bC\in l_0)=1$, and consequently $\pr(\bTheta\in l_0)=1$, holds. 
Before we state the main result of this section, recall that $\|\bx\|=\max_{\bi\in \Z^d} |x_{\bi}|$ for every $\bx=(x_{\bi})_{\bi\in\Z^d}\in l_0$.

\tcb{
\begin{proposition}\label{prop:MA_Q}
Assume $\bC=(C_{\bi})_{\bi\in\Z^d}$ is an arbitrary element of $l_0$ and $\alpha>0$ such that 
\begin{align*}
c:=\sum_{\bk\in \Z^d} \ex\big[ |C_{\bk}|^{\alpha}\big] \in (0,\infty) \, .
\end{align*}
Moreover, assume that $\kappa$ is a $\{-1,1\}$-valued random variable independent of $\bC$.
\begin{itemize}
\item[(i)] Let $\bTheta$ be a random element of $l_0$ with distribution satisfying (\ref{eq:MA_spectral_2}) for all measurable $h:l_0\to \R_+$. Then $\bTheta$ is necessarily a spectral tail process of some stationary regularly varying random field with tail index $\alpha$.
\item[(ii)] Let $\bQ$ be a random element of $l_0$ with distribution 
\begin{align}\label{eq:MA_Q}
\pr(\bQ\in \cdot\, )=\frac{1}{\ex\big[ \|\bC\|^{\alpha}\big]} \ex\left[\ind{\frac{\bC \cdot \kappa }{\|\bC\|}\in \cdot \,} \cdot \|\bC\|^{\alpha}\right] \, .
\end{align} 
If $\bY$ is a tail process in $l_0$ with tail index $\alpha$ and whose spectral tail process $\bTheta$ satisfies (\ref{eq:MA_spectral_2}) for all measurable $h:l_0\to\R_{+}$, then $\bQ$ is one representative of the anchored spectral tail process. Moreover, the extremal index of $\bY$ satisfies
\begin{align*}
\vartheta=\frac{\ex[ \|\bC\|^{\alpha}]}{\sum_{\bk\in \Z^d}\ex\big[ |C_{\bk}|^{\alpha}\big]} \, .
\end{align*} 
\end{itemize}
%
%
\end{proposition}
}

\begin{remark}
Assume that the coefficients in (\ref{eq:MA_process}) are deterministic, that is if for an array $\mathbf{c}=(c_{\bk})_{\bk\in \Z^d}$ and every $\bi\in \Z^d$, $C_{\bi,\bk}= c_{\bk}$ for all $\bk\in \Z^d$. Then $\bC=\mathbf{c}$ and thus since $\|\bC\|=\|\mathbf{c}\|$ is constant, (\ref{eq:MA_Q}) reduces to 
$$\bQ\eind \kappa \cdot \frac{\mathbf{c}}{\|\mathbf{c}\|} \, ,$$ 
which coincides with~\cite[Example 3.1]{basrak:planinic:2020}.
\demo
\end{remark}

\begin{proof}[Proof of \Cref{prop:MA_Q}] \tcb{(i) Observe that $\pr(|\Theta_{\bo}|=1)=1$ holds by construction. Moreover, direct calculation shows that $\bTheta$ satisfies the time-change formula (\ref{eq:time-change0}) for all measurable $h:l_0\to\R_{+}$; we omit the details. Thus, the conclusion follows from \Cref{prop:TCF_another_view} and \Cref{cor:tail_pr_charact}.}

(ii) By \Cref{prop:charact_anch_spectral}, to show that $\bQ$ is a representative it is enough to check that (\ref{eq:palmSpectral_variant2_2}) holds
for all measurable $h:l_0\to \R_+$ which are shift-invariant. Let $h$ be an arbitrary such function, and define 
\begin{align*}
\tilde{h}(\bx)=h\left(\frac{\bx}{\|\bx\|}\right)\cdot  \frac{\|\bx\|^{\alpha}}{\sum_{\bi\in\Z^d}|x_{\bi}|^{\alpha}} \, , \, \bx=(x_{\bi})_{\bi\in \Z^d} \in l_0\setminus\{\bo\} \, .
\end{align*}
 Using (\ref{eq:MA_spectral_2}) for $\tilde{h}$ we have
\begin{align*}
\ex\left[h\left(\frac{\bTheta}{\|\bTheta\|}\right) \cdot \frac{\|\bTheta\|^{\alpha}}{\sum_{\bk \in \Z^d}|\Theta_{\bk}|^{\alpha}}\right]  &= \ex[\tilde{h}(\bTheta)] \\
&= c^{-1}  \ex\Big[ \tsum_{\bk\in \Z^d}\tilde{h}\Big(\frac{\shift{\bk}\bC \cdot \kappa}{|C_{\bk}|}\Big)\cdot |C_{\bk}|^{\alpha} \Big] \\
&= c^{-1}  \ex\Big[ \tsum_{\bk\in \Z^d}h\Big(\frac{\shift{\bk}\bC \cdot \kappa}{\|\bC\|}\Big) \frac{\|\bC\|^{\alpha}}{\sum_{\bj\in \Z^d} |C_{\bj}|^{\alpha}} \cdot |C_{\bk}|^{\alpha} \Big] \\
&= c^{-1}  \ex\Big[ h\Big(\tfrac{\bC \cdot \kappa}{\|\bC\|}\Big)\cdot \|\bC\|^{\alpha}  \cdot \tfrac{\sum_{\bk\in \Z^d}|C_{\bk}|^{\alpha}}{\sum_{\bj\in \Z^d} |C_{\bj}|^{\alpha}} \Big] \\
&= \tfrac{\ex[\|\bC\|^{\alpha}]}{c} \cdot \ex[h(\bQ)] \, ,
\end{align*}
where shift-invariance of $h$ was used to obtain the fourth equality. Taking $h\equiv 1$ yields
\begin{align*}
\frac{\ex\big[\|\bC\|^{\alpha}\big]}{c} = \ex\left[\frac{\|\bTheta\|^{\alpha}}{\sum_{\bk \in \Z^d}|\Theta_{\bk}|^{\alpha}}\right]\stackrel{(\ref{eq:extremal_repre_spectral})}{=}\vartheta \, ,
\end{align*}
and consequently that (\ref{eq:palmSpectral_variant2_2}) holds.
\end{proof}

Observe that one can write (\ref{eq:MA_spectral_2}) as 
\begin{align}\label{eq:MA_sprectral_3}
\pr(\bTheta \in \cdot \,)= \pr^*\Big(\frac{\shift{T}\bC \cdot \kappa}{|C_{T}|} \in \cdot \, \Big) \, ,
\end{align}
where $\pr^*$ is the probability satisfying
\begin{align}\label{eq:MA_spectral_3_chango_of_measure}
\pr^*(B)=c^{-1} \ex\left[ \tsum_{\bk\in \Z^d} |C_{\bk}|^{\alpha} \1{B} \right] \, ,
\end{align}
and $T=T(\bC)$ a random element of $\Z^d$ with distribution 
\begin{align}\label{eq:MA_spectral_3_T}
\pr^*(T=\bk \mid \bC)=\frac{|C_{\bk}|^{\alpha}}{\sum_{\bj\in \Z^d} |C_{\bj}|^{\alpha}} \, , \, \bk\in \Z^d \, .
\end{align} 
Analogous observation can also be made for the distribution of $\bQ$ in (\ref{eq:MA_Q}).

\begin{example}\label{exa:MA}
Let $(Z_{i})_{i\in \Z}$ be an i.i.d.\ sequence of $\mathrm{Pareto}(\alpha)$-distributed random variables for some $\alpha>0$. In particular, (\ref{eq:Z_RV}) holds with the same $\alpha$, and $p=1$ since $Z_i$'s are nonnegative. Further, let $(\varepsilon_i)_{i\in \Z}$ be an i.i.d.\ sequence of Bernoulli random variables such that $\pr(\varepsilon_0=0)=\pr(\varepsilon_0=1)=\tfrac12$, and assume $(\varepsilon_i)_i$ and $(Z_i)_i$ are independent. Finally, let $b>0$ be arbitrary and define a stationary moving average process $\bX=(X_i)_{i\in \Z}$ by
\begin{align*}
X_i=Z_i+  \varepsilon_i b Z_{i-1} \, , \, i\in \Z \, .
\end{align*} 
Note that (\ref{eq:MA_process}) holds with $C_{i,0}=1, C_{i,1}= \varepsilon_i b$ and $C_{i,k}=0, k\neq 0,1$, for all $i\in \Z$. 
\begin{enumerate}[leftmargin=*]
\item (\textbf{Tail process}) By \Cref{rem:moment_assump_mdep}, $(X_i)_i$ admits a tail process, and we will show that its spectral process $\bTheta=(\Theta_i)_{i\in \Z}$ satisfies $\Theta_i=0$ for $|i|\geq 2$ and
\begin{align}\label{eq:MA_example1_spectral}
(\Theta_{-1}, \Theta_0,\Theta_1) = \begin{cases}
(0,1,0) \, , & \text{w.p.} \; \frac{1}{2+b^\alpha} \\
(0,1,b) \, , & \text{w.p.} \; \frac{1}{2+b^\alpha} \\
(1/b,1,0) \, , & \text{w.p.} \; \frac{b^\alpha}{2+b^{\alpha}} \, .
\end{cases}
\end{align}

Indeed, the process $\bC=(\dots, C_{-1}, C_0, C_1, C_2, \dots)$ defined in (\ref{eq:C_new}) is given by
\begin{align*}
\bC=(\dots, 0, 1, \varepsilon_1 b, 0,\dots) \, ,
\end{align*}
that is,
\begin{align*}
\bC = \begin{cases}
 (\dots,0,1,0,0,\dots)\, , & \text{w.p.} \; \frac12 \\
(\dots,0,1,b,0,\dots) \, , & \text{w.p.} \; \frac12  \, .
\end{cases}
\end{align*}
In particular, the constant from (\ref{eq:moment_assump_cor}) equals
\begin{align*}
c=\sum_{k\in \Z} \ex\big[ |C_{k}|^{\alpha}\big] = \frac12 \cdot 1^{\alpha} + \frac12 \cdot (1 ^{\alpha} + b^{\alpha}) = 1 + \frac{b^\alpha}{2}  \, .
\end{align*}
moreover, if $\pr^*$ is the probability from (\ref{eq:MA_spectral_3_chango_of_measure}), one has
\begin{align*}
&\pr^*(\bC=(\dots,0,1,0,\dots))=\pr^*(\varepsilon_1=0)=c^{-1} \pr(\varepsilon_1=0) \cdot 1 = \frac{1}{2+b^\alpha} \, , \\
 &\pr^*(\bC=(\dots,0,1,b,0,\dots))=\pr^*(\varepsilon_1=1)=c^{-1} \pr(\varepsilon_1=1) \cdot (1 ^{\alpha} + b^{\alpha}) = \frac{1+b^\alpha}{2+b^\alpha} \, .
\end{align*}
Finally, given that $\bC=(\dots,0,1,b,0,\dots)$, the random index $T$ from (\ref{eq:MA_spectral_3_T}) satisfies
\begin{align*}
T=\begin{cases}
0 \, , & \text{w.p.} \; \frac{1}{1+b^\alpha} \\
1 \, , & \text{w.p.} \; \frac{b^\alpha}{1+b^\alpha} \, .
\end{cases}
\end{align*}
Taking evertying into account, (\ref{eq:MA_example1_spectral}) now follows from (\ref{eq:MA_sprectral_3}).
\item (\textbf{Anchored tail process}) We will use \Cref{prop:MA_Q} to determine a representative of the anchored spectral tail process. Here one needs to consider cases $b\leq 1$ and $b>1$ separately.
\begin{itemize}
\item[(i)] If $b\leq 1$, one has $\|\bC\|=1$ almost surely. Thus, \Cref{prop:MA_Q} gives that $\bQ:=\bC$ is one representative of the anchored spectral tail process, and that $\vartheta=c^{-1}=\tfrac{2}{2+b^\alpha}$.
\item[(ii)] If $b>1$, calculation similar to the one given in 1.\ above gives that $\bQ$ from (\ref{eq:MA_Q}) satisfies
 \begin{align}\label{eq:MA_example_Q}
\bQ=(\dots,Q_{-1},Q_{0},Q_1,Q_2,\dots) = \begin{cases}
 (\dots,0,1,0,0,\dots)\, , & \text{w.p.} \; \frac{1}{1+b^\alpha} \\
(\dots,0,1/b,1,0,\dots) \, , & \text{w.p.} \; \frac{b^\alpha}{1+b^\alpha}  \, .
\end{cases}
\end{align}
\end{itemize}
A representative of the anchored tail process is obtained by setting $\bZ:=Y\cdot \bQ$ where $Y$ is $\mbox{Pareto}(\alpha)$-distributed and independent of $\bQ$.
\end{enumerate}
To get some intuition note first that the process $\bX$ is $2$-dependent so, by  \Cref{rem:mdep_mixing}, (\ref{eq:randomized_cluster}) and (\ref{eq:randomized_origin}) hold. We will again use the reasoning from $\Cref{exa:MA_intro}$ -- for large $u$, whenever $Z_i>u$ occurs it implies that $X_i/u\approx Z_i/u$ and $X_{i+1}/u\approx \varepsilon_{i+1} b Z_{i}/u$, that is
\begin{align*}
X_{i+1}/u\approx \begin{cases}
 b Z_i/u , & \text{w.p.} \; \frac12 \\
0 \, , & \text{w.p.} \; \frac12  \, ,
\end{cases}
\end{align*}
see \Cref{fig:MA_intro} for an illustration in the case $b=1$.  In particular, if $b>1$ it might happen that $Z_i<u$ but that $Z_i$ is large enough so that $bZ_i>u$. This intuitively explains why the case $(\Theta_{-1}, \Theta_0,\Theta_1)=(1/b,1,0)$ in (\ref{eq:MA_example1_spectral}) is the most probable one when $b>1$ -- by choosing an exceedance over large $u$ by the process $\bX$ uniformly at random, one is most likely to choose the second observation in a cluster with two extreme observations (the first observation is also large but not necessarily larger than $u$). Similarly, since the typical cluster of extremes of $\bX$ is obtained by choosing uniformly at random a cluster of extremes of $\bX$ (having at least one observation exceeding  a large $u$), the same argument gives intuition on why the two cases in (\ref{eq:MA_example_Q}) are not equally probable when $b>1$.
\demo

\end{example}

\begin{example}\label{exa:MA_counter}
Assume that $(Z_i)_{i\in \Z}$ and $(\varepsilon_i)_{i\in \Z}$ are as in \Cref{exa:MA}, but consider the stationary moving average process $\bX=(X_i)_{i\in \Z}$ defined by
\begin{align*}
X_i=Z_i+  \varepsilon_0  Z_{i-1} \, , \, i\in \Z \, .
\end{align*} 
Note that the process $\bC=(\dots, C_{-1}, C_0, C_1, C_2, \dots)$ defined in (\ref{eq:C_new}) coincides with the one in \Cref{exa:MA} (for the case $b=1$ which is also considered in \Cref{exa:MA_intro} in the introduction). In particular, $\bX$ has the same tail and anchored tail process. However, the asymptotic distribution of the cluster containing and being centered around a randomly chosen exceedance-point of $\bX$ (i.e.\ the limit of the left hand side of (\ref{eq:randomized_origin})) is not equal to the distribution of the tail process.  Indeed, depending on the value of $\varepsilon_0$, $\bX$ is either an i.i.d.\ regularly varying sequence or a simple 2-dependent moving average process with deterministic coefficients.  It easy to verify (simply condition on the value of $\varepsilon_0$ and apply (\ref{eq:randomized_origin})) that the left hand side of (\ref{eq:randomized_origin}) converges weakly to $\pr(\bY'\in \cdot \,)$, where $\bY'=(Y_i')_{i\in \Z}$ satisfies
$Y_i'=0$ for $|i|\geq 2$ and
\begin{align*}
(Y_{-1}', Y_0',Y_1') = \begin{cases}
(0,Y_0',0) \, , & \text{w.p.} \; \frac12 \\
(0,Y_0',Y_0') \, , & \text{w.p.} \; \frac14 \\
(Y_0',Y_0',0) \, , & \text{w.p.} \; \frac14 \, ,
\end{cases}
\end{align*}
where $Y_0'$ is $\mathrm{Pareto}(\alpha)$-distributed, which is not the tail process of $\bX$. \demo
\end{example}

\begin{remark}
In the last two examples, for each $\bi\in \Z^d$ coefficients $\bC^{(\bi)}=(C_{\bi,\bk})_{\bk\in \Z^d}\teind \bC^{(\bo)}$ had the same distribution as $\bC=(C_{\bk,\bk})_{\bk\in \Z^d}$. In general, this is not the case. Consider  the moving average process $\bX=(X_i)_{i\in \Z}$ defined by
\begin{align*}
X_i=Z_i+  \varepsilon_i  Z_{i-1} +  \varepsilon_i  Z_{i-2} \, , \, i\in \Z \, ,
\end{align*} 
where $(Z_i)_{i\in \Z}$ and $(\varepsilon_i)_{i\in \Z}$ are again the same as in \Cref{exa:MA}. In this case,
\begin{align*}
\bC^{(0)}=(\dots, 0, 1, \varepsilon_0, \varepsilon_0,0,\dots) \, ,
\end{align*}
while
\begin{align*}
\bC=(\dots, 0, 1, \varepsilon_1 , \varepsilon_2,0,\dots) \, ,
\end{align*}
where, recall, $\varepsilon_1$ and $\varepsilon_2$ are independent. Thus, the distributions of $\bC$ and $\bC^{(0)}$ differ. \demo
\end{remark}

\begin{remark}

Under the assumptions of \Cref{prop:MA_Q}, let $\bQ$ be a random element of $l_0$ with distribution satisfying (\ref{eq:MA_Q}). Then
\begin{align*}
\ex\left[\tsum_{\bk \in \Z^d}|Q_{\bk}|^{\alpha}\right] = \frac{\sum_{\bk\in \Z^d} \ex\big[ |C_{\bk}|^{\alpha}\big]}{\ex[ \|\bC\|^{\alpha}]}<\infty \, ,
\end{align*} 
Moreover, if $\bTheta$ is a random element of $l_0$ with distribution defined by (\ref{eq:palmSpectral_basic_2}), it is easily checked that $\bTheta$ necessarily  satisfies (\ref{eq:MA_spectral_2}). This gives an alternative proof of of \Cref{prop:MA_Q} -- (i) follows from \Cref{rem:charact}, and (ii) is implied by  \Cref{prop:charact_anch_spectral}. \demo
\end{remark}

\section{Postponed proofs}\label{sec:appendix}
\begin{proof}[Proof of equivalence of (\ref{eq:time-change1}) and (\ref{eq:mecke})]
Assuming that (\ref{eq:time-change1}) holds yields
\begin{align*}
\ex\big[\tsum_{\bk\in \e(\bY)} g(-\bk , \shift{\bk}\bY)\big] & = \ex\big[\tsum_{\bk\in \Z^d} g(-\bk , \shift{\bk}\bY) \ind{|Y_{\bk}|>1}\big]   && \\
& = \tsum_{\bk\in \Z^d} \ex\big[g(-\bk , \shift{\bk}\bY) \ind{|Y_{\bk}|>1}\big]  && \\
& = \tsum_{\bk\in \Z^d} \ex\big[g(-\bk , \bY) \ind{|Y_{-\bk}|>1}\big] && \text{by }(\ref{eq:time-change1}) \\
&= \tsum_{\bk\in \Z^d} \ex\big[g(\bk , \bY) \ind{|Y_{\bk}|>1}\big] && \\
& =  \ex\big[\tsum_{\bk\in \e(\bY)} g(\bk , \bY)\big]  \, .
\end{align*}
On the other hand, for fixed $\bk'\in \Z^d$ and $h:\R^{\Z^d}\to \R_+$, define
 $$g(\bk,\bx)=h(\bx)\ind{|x_{\bo}|>1, \bk=-\bk'} \, , \quad \bx=(x_{\bi})_{\bi \in \Z^d} \in \R^{\Z^d} \, , \, \bk\in \Z^d \, .$$ 
 Equation (\ref{eq:mecke}) then gives 
\begin{align*}
\ex\left[h(\shift{\bk'}\bY)\ind{|Y_{\bk'}|>1}\right] & = \ex\big[\tsum_{\bk\in \e(\bY)} g(-\bk , \shift{\bk}\bY)\big] = \ex\big[\tsum_{\bk\in \e(\bY)} g(\bk , \bY)\big] \\
&= \ex\left[h(\bY)\ind{|Y_{\bo}|>1, |Y_{-\bk'}|>1}\right] = \ex\left[h(\bY)\ind{|Y_{-\bk'}|>1}\right] \, .
\end{align*}
\end{proof}

\begin{proof}[\tcb{Proof of (\ref{eq:campbell_tail})}]
For notational simplicity, we will prove the result for $d=1$; extension to the general case is straightforward. Due to the stationarity of $\bX$,
\begin{align*}
\ex\big[\tsum_{k =1}^n f(k/n, \shift{k}\bX / c_n) \1{\{|X_{k}|>c_n\}}\big] & = \tsum_{k =1}^n \ex\big[ f(k/n, \bX / c_n) \1{\{|X_{0}|>c_n\}}\big] \\
& = \ex\big[ \tsum_{k = 1}^n  f(k/n, \bX / c_n) \1{\{|X_{0}|>c_n\}}\big]\\
&= n\ex\big[\textstyle \int_0^1  f(\lceil nt \rceil / n, \bX / c_n)\, \dx t \, \1{\{|X_{0}|>c_n\}}\big] \\
&= n \pr(|X_{0}|>c_n) \ex\big[\textstyle \int_0^1  f(\lceil nt \rceil / n, \bX / c_n)\, \dx t \mid |X_{0}|>c_n\big] \\
&= n \pr(|X_{0}|>c_n) \ex\big[T_n(\bX/c_n) \mid |X_{0}|>c_n\big] \, ,
\end{align*}
where $T_n(\bx):= \textstyle \int_0^1  f(\lceil nt \rceil / n, \bx)\, \dx t$, $\bx\in \R^{\Z}$. 

By the choice of $(c_n)_n$, $n \pr(|X_{0}|>c_n)\to \tau$. On the other hand, since $f$ is bounded and continuous,  for every $\bx, \bx_{n} \in \R^{\Z}$, $n\in \N$, such that $\bx_n\to \bx$ in $\R^{\Z}$, the dominated convergence theorem implies that $T_n(\bx_n)\to T(\bx)$ in $\R$, where 
$$T(\bx):=  \int_0^1  f(t, \bx)\, \dx t \, , \, \bx\in \R^{\Z} \, . $$ 
Since the random variables $T_n(\bX/c_n)$, $n\in \N$, are uniformly bounded (and thus uniformly integrable), a generalized continuous mapping theorem given in \cite[Theorem 5.5]{billingsley:1968} and the definition of the tail process (\ref{eq:conv_to_tail}) yield
\begin{align*}
\ex\big[T_n(\bX/c_n) \mid |X_{0}|>c_n\big] \to \ex\big[T(\bY)\big] = \int_0^1  \ex\big[f(t, \bY)\big]\, \dx t \, .
\end{align*}
\end{proof}

\begin{proof}[Proof of necessity in \Cref{thm:Y_is_point_stat}]
Fix an arbitrary $\bk\in \Z^d$, $\bk\neq \bo$. The key ingredient of the proof is the construction of a rich enough family of bijective exceedance-maps $\tau_n$, $n\in \Z$. The construction is taken over from the proof of \cite[Theorem 4.1]{thorisson:2007}.

 Order the elements of the line $\Z\bk$ in the natural way: $m\bk<n\bk$ if $m<n$ for all $m,n\in \Z$. For $\bx\in \R^{\Z^d}$, let $\e_{\bk}(\bx)=\e(\bx)\cap \Z\bk$ be the set of all exceedance-points of $\bx$ lying on $\Z\bk$. Further, partition $\R^{\Z^d}_0$ into sets
\begin{align*}
A_1&=\{\bx\in \R^{\Z^d}_0 : \sup\e_{\bk}(\bx)=\infty \; \text{ and } \;\inf\e_{\bk}(\bx)=-\infty\} \, , \\
A_2&=\{\bx\in \R^{\Z^d}_0 : \sup{\e_{\bk}(\bx)}<\infty \; \text{ and } \; \inf\e_{\bk}(\bx)>-\infty\} \, , \\
A_3&=(A_1\cup A_2)^c \, .
\end{align*}

\begin{enumerate}
\item If $\bx\in A_1$, define $\tau_n(\bx)$, $n\in \Z$, exactly as in \Cref{exa:nth_point_to_the_right} with $\e(\bx)$ replaced by $\e_{\bk}(\bx)$.
\item Assume that $\bx\in A_2$ and denote by $N\in \N$ the total number of elements of $\e_{\bk}(\bx)$. Let $\bi_0,\dots, \bi_{N-1}$ be the ordered elements of $\e_{\bk}(\bx)$ and $k\in \{0,\dots, N-1\}$ such that $\bi_k=\bo$. For each $n\in \Z$, set
\begin{align*}
\tau_n(\bx)=\bi_{k+n \; (\mathrm{mod}\; N)} \, .
\end{align*}
\item Assume that $\bx\in A_3$ and, in particular, that $\inf\e_{\bk}(\bx)>-\infty$ and $\sup{\e_{\bk}(\bx)}=\infty$; the case $\inf\e_{\bk}(\bx)=-\infty$ and $\sup{\e_{\bk}(\bx)}<\infty$ is handled completely analogously. Let $\bi_0, \bi_1,\dots$, be ordered elements of $\e_{\bk}(\bx)$ and assume that $k\geq 0$ is such that $\bi_k=\bo$. Take an arbitrary bijection $\pi$ between $\N_0$ and $\Z$ and set $\bi_{m}':=\bi_{\pi^{-1}(m)}$, $m\in \Z$. For each $n\in \Z$, set
\begin{align*}
\tau_n(\bx)=\bi_{\pi(k)+n}' \, . 
\end{align*}
\end{enumerate}
Observe that, for all $n\in \Z$, $\tau_n$ is a bijective exceedance-map, and moreover 
$$\tau_{-n}(\shift{\tau_n}\bx)=-\tau(\bx) \, ,$$ 
for all $\bx\in \R^{\Z^d}_0$ ($\tau_{-n}$ deserves to be called the \textit{inverse} exceedance-map of $\tau_n$). Further, 
\begin{itemize}
\item[--] for $\bx\in A_1\cup A_2$, $|x_{\bk}|>1$ if and only if $\tau_1(\bx)=\bk$;
\item[--] for $\bx\in A_3$, $|x_{\bk}|>1$ if and only if $\tau_n(\bx)=\bk$ for exactly one $n\in \Z$.
\end{itemize}
Using the properties of $\tau_n$, $n\in \Z$, and the exceedance-stationarity of $\bY$, for an arbitrary measurable $h:\R^{\Z^d}_0 \to \R_{+}$,
\begin{align*}
\ex\left[h(\shift{\bk}\bY)\ind{|Y_{\bk}|>1}\right] = &  \ex\left[h(\shift{\tau_1}\bY)\ind{\tau_1(\bY)=\bk, \bY\in A_1\cup A_2}\right] \\
& + \sum_{n\in \Z} \ex\left[h(\shift{\tau_n}\bY)\ind{\tau_n(\bY)=\bk, \bY\in A_3}\right] \\
= &  \ex\left[h(\shift{\tau_1}\bY)\ind{\tau_{-1}(\shift{\tau_1}\bY)=-\bk, \shift{\tau_1}\bY\in A_1\cup A_2}\right] \\
& + \sum_{n\in \Z} \ex\left[h(\shift{\tau_n}\bY)\ind{\tau_{-n}(\shift{\tau_n}\bY)=-\bk, \shift{\tau_n}\bY\in A_3}\right] \\
= &  \ex\left[h(\bY)\ind{\tau_{-1}(\bY)=-\bk, \bY\in A_1\cup A_2}\right] \\
& + \sum_{n\in \Z} \ex\left[h(\bY)\ind{\tau_{-n}(\bY)=-\bk, \bY\in A_3}\right] \\
=& \ex\left[h(\bY)\ind{|Y_{-\bk}|>1}\right] \, .
\end{align*}
\end{proof}

\begin{proof}[Proof that (\ref{eq:TCF_inter1}) implies (\ref{eq:time-change0})]
By monotone convergence
\begin{align*}
\ex[h\left(|\Theta_{\bk}|^{-1} \shift{\bk}\bTheta \right) \ind{\Theta_{\bk}\neq 0}] =\lim_{s\to 0} \ex[\tilde{h}\left(\shift{\bk}\bY \right) \ind{|Y_{\bk}|> s}] \, , 
\end{align*}
where $\tilde{h}(\bx)=h(|x_{\bo}|^{-1} \bx)$, $\bx\in \R^{\Z^d}_0$. Now using (\ref{eq:TCF_inter1}) for $\tilde{h}$ yields
\begin{align*}
\ex[h\left(|\Theta_{\bk}|^{-1} \shift{\bk}\bTheta \right) \ind{\Theta_{\bk}\neq 0}] &=\lim_{s\to 0} s^{-\alpha}\ex[\tilde{h}\left(s\bY \right) \ind{|Y_{-\bk}|> 1/s}] \\
&= \lim_{s\to 0} s^{-\alpha}\ex[h\left(\bTheta \right) \ind{|Y_{-\bk}|> 1/s}] \\
&= \lim_{s\to 0} s^{-\alpha}\ex[h\left(\bTheta \right) \min\{|\Theta_{-\bk}|^{\alpha} s^{\alpha}, 1\}] \\ 
&= \lim_{s\to 0} \ex[h\left(\bTheta \right) \min\{|\Theta_{-\bk}|^{\alpha} , s^{-\alpha}\}] \\
&= \ex[h\left(\bTheta \right) |\Theta_{-\bk}|^{\alpha}] \, ,
\end{align*}
where we have used that $|Y_{\bk}|=|Y_{\bo}| \cdot |\Theta_{\bk}|$ with $|Y_{\bo}|$ being $\mathrm{Pareto}(\alpha)$-distributed and independent of $\Theta_{\bk}$ to get the third equality, and monotone convergence to get the final equality.
\end{proof}

\begin{proof}[Proof of (\ref{eq:anchored_process_cluster})]
If $\bZ$ is an arbitrary representative of the anchored process of $\bY$, \cite[Equation (3.15)]{basrak:planinic:2020} implies
\begin{align}\label{eq:anchored_process_cluster_inter}
\lim_{n\toi} \ex[h(\{c_n^{-1}X_{\bi} : \bi \in \{1,\dots,r_n\}^d\}) \mid \textstyle\max_{\bi \in \{1,\dots,r_n\}^d}|X_{\bi}|> c_n] = \ex[h(\bZ)] \, ,
\end{align}
for every bounded, measurable $h:l_0\to [0,\infty)$ which is shift-invariant (i.e.\ $h(\shift{\bk}\bx)=h(\bx)$ for all $\bx\in l_0, \bk\in \Z^d$) and such that $\pr(\bZ\in \mathrm{disc}\, h)=0$. Note that \cite[Equation (3.15)]{basrak:planinic:2020} is a consequence of \cite[Proposition 3.8]{basrak:planinic:2020} which also assumes that the sequence $(c_n)_n$ satisfies $\pr(|X_{\bo}|>c_n)\sim \tfrac{1}{n}$ as $n\toi$. However, the proof is valid for any $(c_n)_n$ and $(r_n)_n$ satisfying only $c_n,r_n\toi$ and (\ref{eq:AC}).

{Recall that the mapping $\shift{A}$ is shift-invariant, see (\ref{eq:SI_of_anchorShift}) and the discussion after this formula.  Moreover, it is easily checked that $\mathrm{disc}\, \shift{A} \subseteq \mathrm{disc}\, A$ (in fact, equality holds), and that $\mathrm{disc}\, A$ is a shift-invariant subset of $l_0$ (that is, $\bx\in \mathrm{disc}\, A$ implies $\shift{\bk}\bx\in \mathrm{disc}\, A$ for every $\bk\in \Z^d$). Since $\bZ$ is a representative of the anchored process (thus, $\shift{A}\bZ\teind \bZ^A$) and since $\pr(\bY \in \mathrm{disc}\, A)=0$, the latter fact and \Cref{cor:palm}(iii) imply that 
$\pr(\bZ \in \mathrm{disc}\, A)=0$. Thus, if $g:l_0\to [0,\infty)$ is an arbitrary continuous and bounded function, the function $h:=g \circ \shift{A}$ is bounded, shift-invariant, and $\pr(\bZ\in \mathrm{disc}\, h)=0$ since $\mathrm{disc}\, h \subseteq \mathrm{disc}\, \shift{A} \subseteq \mathrm{disc}\, A$. Therefore, (\ref{eq:anchored_process_cluster_inter}) applies and convergence in (\ref{eq:anchored_process_cluster}) now follows by the definition of weak convergence.}
\end{proof}

\begin{proof}[\tcb{Proof of \Cref{prop:randomized_origin_cluster}}]
Let $\overline{F}(x):=\pr(|X_0|>x)$ for $x\geq 0$.

\begin{itemize}
\item[1.] We first prove (\ref{eq:number_of_clusters_to_infty}) and (\ref{eq:randomized_cluster}). Under the stated assumptions, \cite[Lemma 10.1.1]{kulik:soulier:2020} and the continuous mapping theorem imply the convergence
\begin{align}\label{eq:randomized_origin_proof_1}
\frac{1}{n\overline{F}(c_n)} \sum_{j=1}^{k_n} h(c_n^{-1} \bX_{n,j}) \1{\{\|\bX_{n,j}\|>c_n\}} \dto \vartheta \int_{0}^{\infty} \ex[h(y\bQ)\1{\{\|y\bQ\|>1\}}] \alpha y^{-\alpha-1} \dx y \, , 
\end{align}
for every bounded, measurable and continuous $h:l_0\to [0,\infty)$ which is shift-invariant, where $\bQ$ is an arbitrary representative of the anchored spectral tail process. Since $\|\bQ\|=1$ a.s., \Cref{lem:polar_decomposition_Z}(iv) implies that the right hand side above equals $\vartheta\ex[h(\bZ)]$ for an arbitrary representative $\bZ$ of the anchored tail process. In fact, (\ref{eq:randomized_origin_proof_1}) holds for all $h$ which are not necessarily continuous, but satisfy $\pr(\bZ\in \mathrm{disc}\, h)=0$. Now the same argument as in the proof of (\ref{eq:anchored_process_cluster}) above yields that 
\begin{align}\label{eq:randomized_origin_proof_2}
\frac{1}{n\overline{F}(c_n)} \sum_{j=1}^{k_n} g(\shift{A}(c_n^{-1} \bX_{n,j})) \1{\{\|\bX_{n,j}\|>c_n\}} \Pto \vartheta\ex[g(\shift{A}\bZ)] = \vartheta\ex[g(\bZ^A)]\, , 
\end{align}
for every bounded, measurable and continuous $g:l_0\to [0,\infty)$; the limit holds in probability because the limit is a constant.

Applying (\ref{eq:randomized_origin_proof_2}) to $g\equiv 1$ yields (\ref{eq:number_of_clusters_to_infty}), which together with (\ref{eq:randomized_origin_proof_2}) yields
\begin{align}\label{eq:randomized_origin_proof_3}
\frac{1}{|N_n^c|} \sum_{j\in N_n^c} g(\shift{A}(c_n^{-1} \bX_{n,j}))=\frac{1}{|N_n^c|} \sum_{j=1}^{k_n} g(\shift{A}(c_n^{-1} \bX_{n,j})) \1{\{\|\bX_{n,j}\|>c_n\}} \Pto \ex[g(\bZ^A)]\, , 
\end{align}
for every bounded, measurable and continuous $g:l_0\to [0,\infty)$. Since the left hand side above is precisely $\ex\big[g\big(\shift{A}(c_n^{-1} \bX_{n,K_n})\big) \mid X_1,\dots, X_n\big]$, (\ref{eq:randomized_origin_proof_3}) and the dominated convergence theorem yield
\begin{align*}
\ex\big[g\big(\shift{A}(c_n^{-1} \bX_{n,K_n})\big)\big] \to   \ex[g(\bZ^A)]\, , 
\end{align*}
for every bounded, measurable and continuous $g:l_0\to [0,\infty)$, which proves (\ref{eq:randomized_cluster}).

\item[2.] We now turn to (\ref{eq:number_of_exc_to_infty}) and (\ref{eq:randomized_origin}). Essentially, we will generalize \cite[Proposition 10.1.4]{kulik:soulier:2020}, see (\ref{eq:randomized_origin_proof_goal2}) below. Let $\phi:l_0\to [0,\infty)$ be an arbitrary measurable, bounded and continuous function; w.l.o.g.\ assume that $\|\phi\|\leq 1$.

Define an another function  $h_\phi:l_0\to [0,\infty)$ by
$$h_{\phi}(\bx):= \sum_{k\in \e(\bx)} \phi(\shift{k}\bx) = \sum_{k\in \Z}\phi(\shift{k}\bx) \1{\{|x_{k}|>1\}}\, . $$ 
Note that $h_{\phi}$ is measurable, shift-invariant and such that $\pr(\bZ\in \mathrm{disc}\, h_{\phi})=0$, but it is not bounded. However, for every $M>0$ define $h_{\phi}^M(\bx):= \min{\{h_{\phi}^M(\bx), M\}}$, $\bx\in l_0$, and for all $n\in \N$, $M>0$, denote
\begin{align*}
W_n= \frac{1}{n\overline{F}(c_n)} \sum_{j=1}^{k_n} h_{\phi}(c_n^{-1} \bX_{n,j}) \1{\{\|\bX_{n,j}\|>c_n\}} \, ,
\end{align*}
and
\begin{align*}
W_n^M= \frac{1}{n\overline{F}(c_n)} \sum_{j=1}^{k_n} h_{\phi}^M(c_n^{-1} \bX_{n,j}) \1{\{\|\bX_{n,j}\|>c_n\}} \, .
\end{align*}
Convergence (\ref{eq:randomized_origin_proof_1}) implies that for every $M>0$, 
$$W_n^M\dto \vartheta \ex[h_{\phi}^M(\bZ)] \, , \quad \text{as $n\toi$} \, , $$ 
while the monotone convergence theorem yields that 
$$\lim_{M\toi} \vartheta \ex[h_{\phi}^M(\bZ)] = \vartheta \ex[h_{\phi}(\bZ)]  \, . $$
Observe that by the Palm duality (\ref{eq:tail_palm_construction}), the limit is
\begin{align*}
\vartheta \ex[h_{\phi}(\bZ)] = \frac{1}{\ex\abs{\e(\bZ)}} \ex\left[ \tsum_{k \in \e(\bZ)} \phi(\shift{k}\bZ)\right] = \ex[\phi(\bY)] \, .
\end{align*}
Thus, to prove that
\begin{align}\label{eq:randomized_origin_proof_goal}
\frac{1}{n\overline{F}(c_n)} \sum_{j=1}^{k_n} h_{\phi}(c_n^{-1} \bX_{n,j}) \1{\{\|\bX_{n,j}\|>c_n\}} = W_n\dto \ex[\phi(\bY)] \, ,
\end{align}
it is enough to show that
\begin{align}\label{eq:randomized_origin_proof_4}
\lim_{M\toi}\limsup_{n\toi} \ex\big|W_n - W_n ^M\big| = 0 \, .
\end{align}
By stationarity
\begin{align*}
\ex\big|W_n - W_n ^M\big| & \leq \frac{k_n}{n\overline{F}(c_n)} \ex\Big[\big|h_{\phi}(c_n^{-1}\bX_{n,1}) - h_{\phi}^M(c_n^{-1}\bX_{n,1})\big| \1{\{\|\bX_{n,1}\|>c_n\}}\Big] \\
&\leq \frac{k_n}{n\overline{F}(c_n)} \ex\Big[h_{\phi}(c_n^{-1}\bX_{n,1}) \1{\{\|\bX_{n,1}\|>c_n, h_{\phi}(c_n^{-1}\bX_{n,1})>M\}}\Big] \, ,
\end{align*}
where, recall, $\bX_{n,1}=(X_1,\dots, X_{r_n})$.
Since $\|\phi\|\leq 1$, $h_{\phi}(\bx) \leq \sum_{k\in \Z} \1{\{|x_{k}|>1\}} = |\e(\bx)| $ for  all $\bx\in l_0$, and thus
\begin{align*}
\ex\big|W_n - W_n ^M\big| & \leq \frac{k_n}{n\overline{F}(c_n)} \sum_{k=1}^{r_n} \pr(|X_k|>c_n, |\e(c_n^{-1}\bX_{n,1})|>M) \\
& \leq \frac{k_n}{n\overline{F}(c_n)} \sum_{k=1}^{r_n} \pr(|X_k|>c_n, |\e(c_n^{-1}(X_{k-r_n},\dots, X_{k+r_n}))|>M) \\
& \leq \frac{k_n r_n}{n\overline{F}(c_n)} \pr(|X_0|>c_n, |\e(c_n^{-1}(X_{-r_n},\dots, X_{r_n}))|>M) \, .
\end{align*}
Since $\lim_{n\toi}\tfrac{k_n r_n}{n}=1$, and since the map $\bx\mapsto \e(\bx)$ is a.s.\ continuous on $l_0$ w.r.t.\ the distribution of $\bY$, convergence in (\ref{eq:tail_process_cluster}) implies that
\begin{align*}
\limsup_{n\toi} \ex\big|W_n - W_n ^M\big| &\leq \pr(|\e(c_n^{-1}(X_{-r_n},\dots, X_{r_n}))|>M \mid |X_0|>c_n) \\
& = \pr(|\e(\bY)|>M) \, ,
\end{align*}
which holds for all but countably many $M$'s, and (\ref{eq:randomized_origin_proof_4}) now follows since $\pr(|\e(\bY)|<\infty)=1$.

Thus, for every measurable, bounded and continuous function $\phi:l_0\to [0,\infty)$,  (\ref{eq:randomized_origin_proof_goal}) holds, i.e.\
\begin{align}\label{eq:randomized_origin_proof_goal2}
\frac{1}{n\overline{F}(c_n)} \sum_{i=1}^{k_n r_n} \phi(c_n^{-1} \shift{i}\bX_{n}^i) \1{\{|X_i|>c_n\}} \Pto \ex[\phi(\bY)] \, ,
\end{align} 
where for each $n\in \N$ and $i\in \{1,\dots,k_n r_n\}$, $\bX_n^i$ is the block $\bX_{n,j}$ which contains $X_i$. Convergences (\ref{eq:number_of_exc_to_infty}) and (\ref{eq:randomized_origin}) now follow from (\ref{eq:randomized_origin_proof_goal2}) in the same way as (\ref{eq:number_of_clusters_to_infty}) and (\ref{eq:randomized_cluster}) followed from (\ref{eq:randomized_origin_proof_2}).
\end{itemize}
\end{proof}

\section*{Acknowledgements} 
The work of the author was supported by the grant IZHRZ0\_180549 from the Swiss National Science
Foundation and Croatian Science Foundation, project \enquote{Probabilistic and analytical aspects of generalised
regular variation}. 

%



\end{document}